\documentclass[11pt,twoside]{article}
\usepackage{amsmath,amssymb,graphicx}
\pdfoutput=1 
\newtheorem{proposition}{Proposition}[section]
\newtheorem{theorem}[proposition]{Theorem}
\newtheorem{lemma}[proposition]{Lemma}

\newtheorem{example}[proposition]{Example}
\newtheorem{algorithm}[proposition]{Algorithm}
\newtheorem{assumption}[proposition]{Assumption}

\newenvironment{proof}{{\noindent \bf Proof:}}{\hfill $\fbox{}$ \vspace*{5mm}}
\newcommand{\ba}{{\bf a}}

\newcommand{\bx}{{\bf x}}

\newcommand{\la}{\lambda}
\newcommand{\La}{\Lambda}

\newcommand{\ca}{\mathcal{A}}

\newcommand{\cj}{\mathcal{J}}

\newcommand{\cl}{\mathcal{L}}

\newcommand{\cm}{\mathcal{M}}
\newcommand{\cn}{\mathcal{N}}
\newcommand{\co}{\mathcal{O}}
\newcommand{\cp}{\mathcal{P}}
\newcommand{\cq}{\mathcal{Q}}

\newcommand{\cs}{\mathcal{S}}
\newcommand{\cv}{\mathcal{V}}

\newcommand{\ct}{\mathcal{T}}
\newcommand{\cx}{\mathcal{X}}

\newcommand{\cy}{\mathcal{Y}}
\newcommand{\cz}{\mathcal{Z}}

\newcommand{\Diag}{{\rm Diag}}

\newcommand{\grad}{{\rm grad\;}}

\newcommand{\im}{{\rm im}}

\newcommand{\ve}{{\rm vec}}

\newcommand{\tr}{{\rm tr}}

\newcommand{\R}{{\mathbb R}}

\newcommand{\Rn}{{\mathbb R}^n}
\newcommand{\Cnn}{{\mathbb C}^{n\times n}}
\newcommand{\Rnn}{{\mathbb R}^{n\times n}}

\newcommand{\SRn}{\mathbb{SR}^{n\times n}}
\newcommand{\BE}{\begin{equation}}
\newcommand{\EE}{\end{equation}}

\newcommand{\normmm}[1]{{\vert\kern-0.25ex \vert\kern-0.25ex \vert #1
    \vert\kern-0.25ex \vert\kern-0.25ex\vert}}
\begin{document}

\title{A Riemannian Inexact Newton-CG Method for Nonnegative Inverse Eigenvalue Problems: Nonsymmetric Case}
\author{Zhi Zhao\thanks{Department of Mathematics, School of Sciences, Hangzhou Dianzi University,
Hangzhou 310018, People's Republic of China (zzhao@hdu.edu.cn).
The research of this author is supported by the National Natural Science Foundation of China (No. 11601112).}
\and Zheng-Jian Bai\thanks{Corresponding author. School of Mathematical Sciences and Fujian Provincial Key Laboratory on Mathematical Modeling \& High Performance Scientific Computing,  Xiamen University, Xiamen 361005, People's Republic of China (zjbai@xmu.edu.cn). The research of this author is partially supported by the National Natural Science Foundation of China (No. 11671337), the Natural Science Foundation of Fujian Province of China (No. 2016J01035), and the Fundamental Research Funds for the Central Universities (No. 20720150001).}
\and Xiao-Qing Jin\thanks{Department of Mathematics, University of Macau, Macao, People's Republic of China (xqjin@umac.mo). The research of this author is supported by the research grant MYRG2016-00077-FST from University of Macau.} }

\maketitle
\begin{abstract}
This paper is concerned with the nonnegative inverse eigenvalue problem of finding a nonnegative matrix such that its  spectrum is the prescribed self-conjugate set of complex numbers. We first reformulate the nonnegative inverse eigenvalue problem as an under-determined constrained nonlinear matrix equation over several matrix manifolds. Then we propose a Riemannian inexact Newton-CG method for solving the nonlinear matrix equation. The global and quadratic convergence of the proposed method is established under some mild conditions. We also extend the proposed method to the case of prescribed entries. Finally, numerical experiments are reported to illustrate the efficiency of the proposed method.

\end{abstract}

\vspace{3mm}


\vspace{3mm}

\section{Introduction}
An $n$-by-$n$ nonnegative matrix $C$ is a real matrix whose entries are all greater than or equal to zero, i.e., $(C)_{ij}\ge 0$ for all $i,j=1,\ldots,n$, where $(C)_{ij}$ denotes the $(i,j)$th entry of $C$. Nonnegative matrices arise in various applications such as the Markov chain, linear complementary problems, probabilistic algorithms, discrete distributions, categorical data, group theory, matrix scaling,
and economics. See for instance \cite{BR97,BP79,M88,SA06} and the references therein.

In this paper, we consider the following nonnegative inverse eigenvalue problem (NIEP):

{\bf NIEP.} {\em Given a self-conjugate set of complex numbers $\{\lambda_1, \lambda_2, \ldots , \lambda_n\}$,
find an $n$-by-$n$ nonnegative matrix $C$ such that its eigenvalues are
$\lambda_1, \lambda_2, \ldots , \lambda_n$.}

The early works on the NIEP are due to Sule\textsc{\u{i}}manova \cite{S49}, Karpelevi\u{c} \cite{K51}, and Perfect \cite{P53,P55}.
There has been much literature on the study of the NIEP since then. On the solvability conditions of the NIEP, one may refer to \cite{BJ84,BH91,ELN04,FM79,LS06,LL78,O83,R96,S03,S06-2,S13}. For more comprehensive discussions on the NIEP, one may refer to \cite{CG02,CG05,M88,X98} and the references therein. There are a few numerical methods for solving the NIEP such as
the constructive method \cite{S83}, the alternating projection method \cite{O06}, isospectral gradient flow methods \cite{CL11,CDS04,CD91,CG98}, and a fast recursive algorithm \cite{L13} for the case where the prescribed eigenvalues are all real and satisfy  an additional inequality.

Recently, there exists some literature on Riemannian optimization methods for eigenvalue problems and inverse eigenvalue problems. See for instance \cite{AMS08,B08,S94,YBZC16,ZBJ14,ZJB16}. In this paper,  we propose a Riemannian inexact Newton-CG method for solving the NIEP. This is motivated by the recent two papers due to Dedieu, Priouret, and Malajovich \cite{DPM03} and Simons \cite{S06}. In \cite{DPM03}, based on the exponential map, Dedieu  et. al. presented Newton's method  for finding zeros of a mapping from  a Riemannian manifold  to  a linear space of the same dimension and the quadratic convergence was also investigated. In \cite{S06}, Simons gave some inexact Newton methods for solving an under-determined system of nonlinear equations over vector spaces.  By using the real Schur decomposition of a real square matrix,  we  rewrite the NIEP as an equivalent under-determined constrained nonlinear matrix equation over several matrix manifolds.   Then we present  a Riemannian inexact Newton-CG method for solving the under-determined constrained nonlinear matrix equation.
Under some mild conditions, the global and quadratic convergence property of the proposed method is established.
We also extend the proposed method to the case of prescribed entries. Numerical experiments show that the proposed method is more efficient
than the alternating projection method in \cite{O06} and the Riemannian nonlinear conjugate gradient methods in \cite{YBZC16,ZJB16}.

Throughout this paper, we use the following notations. The symbols $A^T$ and $A^H$ denote the transpose and  complex conjugate transpose of a matrix $A$ respectively.
$I_n$ is the identity matrix of order $n$. Let $\Rnn$ and $\SRn$ be the set of all $n$-by-$n$  real matrices and the set of all $n$-by-$n$ real symmetric matrices, respectively. Let $\Rnn_{+}$ and $\SRn_+$ denote the nonnegative orthants of $\Rnn$ and $\SRn$, respectively. For two matrices $A,B\in \Rnn$, $A \odot B$ and $[A,B]:=AB-BA$ mean  the Hadamard product and Lie Bracket of $A$ and $B$, respectively. Given a vector $\ba\in\Rn$, $\Diag(\ba)$ denotes a diagonal matrix with $\ba$ on its diagonal. Let $\ve(A)$ be the vectorization of a matrix $A$, i.e., a column vector obtained by stacking the columns of $A$ on top of one another. Denote by $\tr(A)$ the sum of the diagonal entries of a square matrix $A$. Define the index set $\cn:=\{(i,j) \ | \ i,j=1,\ldots,n\}$.
For two finite-dimensional vector spaces $\cx$ and $\cy$ equipped with a scalar inner product $\langle\cdot,\cdot\rangle$ and its induced norm $\|\cdot\|$, let $\ca:\cx\to \cy$ be a linear operator. The adjoint operator of $\ca$ is denoted by $\ca^*$.
The operator norm of $\ca$ is defined by $\normmm{\ca}:=\sup\{ \| \ca x\| \ | \ x\in\cx\mbox{ with } \|x\|=1\}.$

The rest of this paper is organized as follows. In section \ref{sec3} we propose a Riemannian inexact Newton-CG method for solving the NIEP.
In section \ref{sec4} the global and quadratic convergence of the proposed method is established under some mild conditions.
In section \ref{sec5}, we discuss  some extensions. Finally, some numerical tests are reported in section \ref{sec6}
and we give some concluding remarks in section \ref{sec7}.
\section{Riemannian inexact Newton-CG method}\label{sec3}
In this section, we first reformulate the NIEP as a nonlinear matrix equation defined on a Riemannian product manifold. Then we propose a Riemannian  inexact Newton-CG method for solving the nonlinear matrix equation.
\subsection{Reformulation}
For the two matrix sets $\Rnn_{+}$ and $\Rnn$, we have
\[
\Rnn_{+} = \big\{S\odot S \ | \  S \in \Rnn \big\}.
\]
Notice that the set of prescribed eigenvalues $\{\lambda_1, \lambda_2, \ldots , \lambda_n\}$ is closed under complex conjugation. Without loss of generality, we can assume
\[
\la_{2i-1}=a_i+ b_i\sqrt{-1}, \quad \la_{2i}=a_i-b_i\sqrt{-1}, \quad i=1,\ldots,s; \quad \la_i\in\R, \quad i=2s+1,\ldots,n.
\]
where $a_i, b_i\in\R$ with $b_i\neq 0$ for $i=1,\ldots,s$. Define the following block diagonal matrix
\[
\La:={\rm blkdiag}\left(\la_1^{[2]}, \ldots,\la_s^{[2]},\la_{2s+1},\ldots,\la_n\right),
\]
where
\[
\lambda_i^{[2]} :=
\left[
\begin{array}{cc}
a_i & b_i \\
-b_i & a_i
\end{array}
\right],\quad i=1,\ldots,s.
\]
By using the real Schur decomposition for a real square matrix \cite{GV13}, the set of all isospectral matrices can be defined
as the following matrix set:
\[
\cm(\Lambda): =  \big\{ X \in \Rnn \ | \ X = Q(\Lambda + V) Q^T ,\;  Q\in \co(n) , V\in \cv \big\}.
\]
Here,  $\co(n)$ means the set of all $n$-by-$n$ orthogonal matrices, i.e.,
\[
\co(n): =  \big\{ Q \in \Rnn \ | \ Q^TQ = I_{n} \big\}
\]
and the set $\cv$ is defined by
\[
\cv : =  \big\{ V \in \Rnn  \ | \  V_{ij}= 0, \; (i,j) \in \mathcal{I} \big\},
\]
where
\[
\mathcal{I}: = \big\{ (i,j) \ | \ i\geq j \; \mbox{or} \; \Lambda_{ij} \neq 0, \; i,j =1,\ldots,n \big\}\subset\cn.
\]
Thus the NIEP has a solution if and only if $\cm(\Lambda) \cap \Rnn_{+} \neq \emptyset$.

Suppose that the NIEP has at least one solution. Then the NIEP aims to solve
the following constrained nonlinear matrix equation:
\BE\label{eq:NIEP1}
G(S,Q,V)= \mathbf{0}_{n\times n}
\EE
for $(S,Q,V)\in\Rnn \times \co(n) \times \cv$, where $\mathbf{0}_{n\times n}$ is the zero matrix of order $n$.
The smooth mapping $G:\Rnn \times \co(n) \times \cv \to\Rnn$ is defined by
\[
G(S,Q,V) : = S \odot S - Q(\Lambda + V)Q^T,\quad (S,Q,V)\in\Rnn \times \co(n) \times \cv.
\]

We point out that $G$ is a smooth mapping from the product manifold $\Rnn\times \co(n)\times \cv$ to the linear space $\Rnn$. Once we find a solution $(\overline{S},\overline{Q},\overline{V})\in\Rnn \times \co(n) \times \cv$ to the  nonlinear equation (\ref{eq:NIEP1}), then the matrix $\overline{C}:=\overline{S}\odot \overline{S}$ is a solution to the NIEP.

\subsection{Riemannian inexact Newton-CG method}

In \cite{S06}, Simons presented some inexact Newton methods for the  under-determined system of nonlinear equations $F(\bx)={\bf 0}_n$, where $F:\R^m\to\Rn$ is continuously differentiable ($m>n$) and $\mathbf{0}_{n}$ is an $n$-vector of all zeros. Sparked by this, in this section, we propose a Riemannian  inexact Newton-CG method for solving the  nonlinear equation (\ref{eq:NIEP1}).

We first note that $\Rnn\times \co(n)\times \cv$ is a product manifold and as shown in Appendix A, the nonlinear matrix equation (\ref{eq:NIEP1}) is under-determined for all $n\ge 2$. It is easy to see  that  $\Rnn \times \co(n) \times \cv$ is an embedded submanifold of $\Rnn \times \Rnn \times \Rnn$ and then every  tangent space $T_{(S,Q,V)}( \Rnn \times\co(n)\times \cv)$, which is characterized as in Appendix A, can be regarded as  a subspace of $T_{(S,Q,V)}( \Rnn \times\Rnn\times \Rnn)\simeq \Rnn \times\Rnn\times \Rnn$, where ```$\simeq$" means the identification of two sets. Hence, the Riemannian metric of $\Rnn \times\co(n)\times \cv$ inherited from the standard inner product on $\Rnn \times \Rnn \times \Rnn$ is given by
\begin{eqnarray}\label{def:ip}
g_{(S,Q,V)}\big( (\xi_1, \zeta_1, \eta_1), (\xi_2, \zeta_2,\eta_2) \big)&:=&\langle(\xi_1, \zeta_1, \eta_1), (\xi_2, \zeta_2,\eta_2)\rangle \nonumber\\ &:=&\tr(X_1^TX_2)+\tr(Y_1^TY_2)+\tr(Z_1^TZ_2),
\end{eqnarray}
for all $(S,Q,V) \in \Rnn \times \co(n) \times \cv$ and $(\xi_1, \zeta_1,\eta_1), (\xi_2, \zeta_2,\eta_2) \in
T_{(S,Q,V)}\big(\Rnn \times \co(n) \times \cv\big)$. In what follows, we denote by $\langle\cdot,\cdot\rangle$ and $\|\cdot\|$  the Riemannian metric and its induced norm on  $\Rnn \times \co(n) \times \cv$ respectively.

Next, we propose  a Riemannian  inexact Newton-CG method for solving the under-determined matrix equation (\ref{eq:NIEP1}).
As in \cite{DPM03}, one may propose the following geometric Newton method: Given the current iterate $X^k := (S^k,Q^k,V^k) \in \Rnn \times \co(n) \times \cv$, solve the Newton equation:
\BE\label{eq:NNEW1}
\mathrm{D}G(X^k)[\Delta X^k ] = -G(X^k)
\EE
for $\Delta X^k:=(\Delta S^k,\Delta Q^k, \Delta V^k)\in T_{X^k}( \Rnn \times\co(n)\times \cv)$ and set
\[
X^{k+1}:= R_{X^k} (\Delta X^k),
\]
where $\mathrm{D}G(X^k)$ is the differential of $G$ at $X^k$ and $R$ is a retraction on $\Rnn\times\co(n)\times \cv$. On the explicit expressions of $\mathrm{D}G(\cdot)$ and $R$, one may refer to Appendix A.

We see that (\ref{eq:NNEW1}) is under-determined, which may have many solutions.
Sparked by the idea in \cite{S06,WW90}, the minimum norm  solution of (\ref{eq:NNEW1}) is given by:
\[
\Delta X^k=(\mathrm{D}G(X^k))^\dag G(X^k),
\]
where $(\mathrm{D}G(X^k))^{\dag}$ means the pseudoinverse of $\mathrm{D}G(X^k)$ \cite{L69}.
In particular, if the linear operator $\mathrm{D}G(X^k)$ is surjective,
then we have \cite[Chap. 6]{L69}:
\[
(\mathrm{D}G(X^k))^{\dag} = (\mathrm{D}G(X^k))^*\circ \big(\mathrm{D}G(X^k)\circ(\mathrm{D}G(X^k))^*\big)^{-1}.
\]
In this case, one may solve the following normal equation:
\BE\label{eq:NPNEW0}
\mathrm{D}G(X^k)\circ (\mathrm{D}G(X^k))^*[ \Delta Z ] = -G(X^k),\quad \text{s.t.} \quad \Delta Z^k \in T_{G(X^k)} \Rnn
\EE
for the minimum norm solution $\Delta X^k = (\mathrm{D}G(X^k))^* [ \Delta Z^k]\in T_{X^k}(\Rnn \times\co(n)\times \mathcal{V})$, where $\mathrm{D}G(X^k)^*$ is  the adjoint of
$\mathrm{D}G(X^k)$ with respect to the Riemannian metric $\langle\cdot,\cdot\rangle$ on  $\Rnn \times \co(n) \times \cv$. For the explicit expression of $\mathrm{D}G(\cdot)^*$, one may refer to Appendix A. Thus, the conjugate gradient (CG) method \cite{GV13} can be used to solve the self-adjoint and positive definite equation (\ref{eq:NPNEW0}).

We note that $\mathrm{D}G(X^k)\circ (\mathrm{D}G(X^k))^*$  may be ill-conditioned or singular.
Instead of (\ref{eq:NPNEW0}), one may solve the following perturbed normal equation:
\[
\Big(\mathrm{D}G(X^k)\circ (\mathrm{D}G(X^k))^* + \overline{\sigma}\mathrm{id}_{T_{G(X^k)} \Rnn}\Big)[ \Delta Z^k ] = -G(X^k)
\]
for $\Delta Z^k \in T_{G(X^k)} \Rnn\simeq\Rnn$, where $\overline{\sigma} >0$ is a prescribed constant and $\mathrm{id}_{T_{G(X^k)}}$ denotes the identity operator on $T_{G(X^k)}\Rnn$.

Based on the above discussion, we propose the following Riemannian inexact Newton-CG algorithm for solving  (\ref{eq:NIEP1}).

\begin{algorithm} \label{nm1}
 {\rm (\bf{Riemannian inexact Newton-CG method})}
\begin{description}
\item [{Step 0.}] Choose an initial point $X^0 \in \Rnn \times \co(n) \times \cv$, $\overline{\sigma}_{\max}, \overline{\eta}_{\max}, \widehat{\eta}_{\max} \in [0,1)$, $t\in (0,1)$, $0<\theta_{\min}<\theta_{\max}<1$. Let $k:=0$.

\item [{Step 1.}] Apply the CG method to solving
    \BE\label{eq:le}
    \big(\mathrm{D}G(X^k)\circ (\mathrm{D}G(X^k))^* +\overline{\sigma}_k \mathrm{id}_{T_{G(X^k)} \Rnn} \big)[\Delta Z^k] = - G(X^k),
    \EE
    for $\Delta Z^k \in T_{G(X^k)} \Rnn$ such that
    \BE\label{eq:tol1}
    \|\big( \mathrm{D}G(X^k) \circ (\mathrm{D}G(X^k))^* +\overline{\sigma}_k \mathrm{id}_{T_{G(X^k)} \Rnn}\big)[ \Delta Z^k ]
    + G(X^k)\big\|_{F} \le  \overline{\eta}_k\|  G(X^k) \|_{F} ,
    \EE
    and
    \BE\label{eq:tol2}
    \|\mathrm{D}G(X^k) \circ (\mathrm{D}G(X^k))^* [ \Delta Z^k ] + G(X^k)\|_{F} \leq  \widehat{\eta}_{\max}\|  G(X^k) \|_{F} ,
    \EE
    where $\overline{\sigma}_k:=\min\{\overline{\sigma}_{\max},\| G(X^k)\|_{F}\}$, $\overline{\eta}_k:=\min\{\overline{\eta}_{\max},\| G(X^k)\|_{F}\}$. Then let
    \BE\label{eq:direstep}
    \widehat{\Delta X}^k= (\mathrm{D}G(X^k))^*[ \Delta Z^k ],\quad
    \widehat{\eta}_k:= \frac{\|\mathrm{D}G(X^k) [ \widehat{\Delta X}^k ] + G(X^k)\|_{F}}{\| G(X^k)\|_{F}}.
    \EE

\item [{Step 2.}] Evaluate $G\big(R_{X}^k(\widehat{\Delta X}^k)\big)$.
                  Set $\eta_k = \widehat{\eta}_k$ and $\Delta X^k = \widehat{\Delta X}^k$.

                 {\bf Repeat} until $\|G\big(R_{X^k}(\Delta X^k)\big)\|_F\le (1-t(1-\eta_k)) \|G(X^k)\|_F$.

                  \qquad\quad Choose $\theta \in [\theta_{\min},\theta_{\max}]$.

                  \qquad\quad Replace  $\Delta X^k$ by $\theta \Delta X^k $ and $\eta_k$ by $1-\theta(1-\eta_k)$.

                  {\bf end (Repeat)}

                  Set
                  \[
                  X^{k+1} := R_{X^k} (\Delta X^k ).
                  \]
\item [{Step 3.}] Replace $k$ by $k+1$ and go to {\bf Step 1}.
\end{description}
\end{algorithm}

We note that, in \cite{DPM03}, the new iterate for Newton's method is updated by using the exponential map while in Algorithm \ref{nm1}, the new iterate $X^{k+1}$ is updated by using a retraction on $\Rnn \times \co(n) \times \cv$ defined as in Appendix A instead of the exponential map on $\Rnn \times \co(n) \times \cv$, which is in general computationally costly \cite[p.59 and p.103]{AMS08}. In addition, in {\bf Step 2} of Algorithm \ref{nm1}, one needs to choose a scaling factor $\theta \in [\theta_{\min},\theta_{\max}]$. As in \cite{S06}, one may choose  $\theta$ by employing the quadratic backtracking method (see also \cite{DS96}).
Let
\[
T(\Rnn \times \co(n) \times \cv)=\cup_{X\in\Rnn\times \co(n)\times\cv}T_X(\Rnn \times \co(n) \times \cv)
\]
be the tangent bundle of $\Rnn \times \co(n) \times \cv$ \cite[p.36]{AMS08}. The pullback $\widehat{G}$ of $G$ is
a smooth mapping from $T(\Rnn \times \co(n) \times \cv)$ to $\Rnn$ defined by
\BE \label{def:pb}
\widehat{G} (\xi):= G(R(\xi)), \quad \forall \xi\in T(\Rnn \times \co(n) \times \cv).
\EE
The restriction of $\widehat{G}$ on $T_X(\Rnn \times \co(n) \times \cv)$ for $X\in \Rnn \times \co(n) \times \cv$ is defined by
\[
\widehat{G}_X(\xi_X)= G(R_X(\xi_X)), \quad \forall \xi_X \in T_X(\Rnn \times \co(n) \times \cv).
\]
Then one has
\BE\label{s1:e2}
\mathrm{D}G(X) = \mathrm{D}\widehat{G}_X(0_X),\quad \forall X\in \Rnn \times \co(n) \times \cv,
\EE
where $0_X$ is the origin of $T_X(\Rnn \times \co(n) \times \cv)$.
We now find an approximate minimizer of the cost function
\[
u(\theta) : = \|G(R_{X^k}(\theta\Delta X^k))\|_F^2 = \|\widehat{G}_{X^k}(\theta\Delta X^k)\|_F^2.
\]
Define a quadratic polynomial by
\[
q(\theta) := (u(1) - u(0) - u'(0))\theta^2 + u'(0)\theta + u(0),
\]
where
\[
\begin{array}{l}
u(0) = \|\widehat{G}_{X^k}(0_{X^k})\|_F^2=\|G(X^k)\|_F^2, \qquad
u(1) = \|\widehat{G}_{X^k}(\Delta X^k)\|_F^2=\|G(R_{X^k}(\Delta X^k))\|_F^2 , \\[2mm]
u'(0) = 2\langle {\rm D}\widehat{G}_{X^k}(0_{X^k})[\Delta X^k], \widehat{G}_{X^k}(0_{X^k}) \rangle =2\langle {\rm D}G(X^k)[\Delta X^k], G(X^k) \rangle.
\end{array}
\]
Obviously, the values of $u(0)$ and $u(1)$ have been evaluated in  Algorithm \ref{nm1} and it is not so complicated to compute $u'(0)$.
It is easy to check that
\[
q'(\theta) = 2(u(1) - u(0) - u'(0))\theta + u'(0)\quad\mbox{and}\quad
q''(\theta) = 2(u(1) - u(0) - u'(0)).
\]
If $q''(\theta) \leq 0$, then the quadratic polynomial $q$ is concave and we choose $\theta = \theta_{\max}$.
If $q''(\theta) > 0$, then the minimizer of $q$ is reached at the point $\theta$ satisfying $q'(\theta) = 0$, i.e.,
\[
\theta = \frac{-u'(0)}{2(u(1) - u(0) - u'(0))}.
\]
Since we require $\theta \in [\theta_{\min},\theta_{\max}]$,
the approximate minimizer $\theta$ of $u$ is given by
\[
\theta = \min\left\{ \max\Big\{\theta_{\min},\frac{-u'(0)}{2(u(1) - u(0) - u'(0))}\Big\}, \theta_{\max} \right\}.
\]

\section{Convergence analysis}\label{sec4}

In this section, we establish the global and quadratic convergence of Algorithm \ref{nm1}.
Notice that $\Rnn$ and $\cv$ are two linear matrix manifolds and $\co(n)$ is a compact manifold. For the retraction $R$ on $\Rnn \times \co(n) \times \cv$ defined as in Appendix A, there exist two scalars $\nu >0$ and $\mu_{\nu} >0$ such that \cite[p.149]{AMS08},
\BE\label{retr:bd-2}
\nu\| \Delta X \| \geq  {\rm dist}\big(X,R_{X} (\Delta X) \big),
\EE
for all $X:=(S,Q,V)\in\Rnn \times \co(n) \times \cv$ and
\[
\Delta X:=(\Delta S, \Delta Q,\Delta V)\in T_{X}(\Rnn \times \co(n) \times \cv)
\]
with $\| \Delta X\|\leq \mu_{\nu}$, where ``{\rm dist}" means the Riemannian distance on $\Rnn \times \co(n) \times \cv$.

\subsection{Global convergence}
To prove the global convergence of Algorithm \ref{nm1}, we need some preliminary lemmas. On  the iterate $\widehat{\Delta X}^k$ generated by Algorithm \ref{nm1}, we have the following estimate.

\begin{lemma}\label{lemma:direnorm}
Assume that $\mathrm{D}G(X^k):T_{X^k}(\Rnn \times\co(n)\times \mathcal{V})\to T_{G(X^k)}\Rnn$ is surjective for all $k$. If the linear matrix equation {\rm (\ref{eq:le})} is solvable  such that conditions {\rm (\ref{eq:tol1})} and {\rm (\ref{eq:tol2})}  are satisfied for all $k$, then one has for all $k$,
\[
\| \widehat{\Delta X}^k \| \leq (1+ \overline{\eta}_{k})\normmm{( \mathrm{D}G(X^k))^\dag} \cdot \|G(X^k)\|_F.
\]
\end{lemma}

\begin{proof}
Let
\[
\mathrm{id}:=\mathrm{id}_{T_{G(X^k)}}, \; J(X^k):=\mathrm{D}G(X^k) \circ (\mathrm{D}G(X^k))^* +\overline{\sigma}_k \mathrm{id},
\; V(X^k):=G(X^k) +J(X^k)[ \Delta Z^k ].
\]
We get by (\ref{eq:tol1}),
\BE\label{ieq:vg}
\| V(X^k) \|_F  \leq \overline{\eta}_k \| G(X^k) \|_F.
\EE
By the assumption that $\mathrm{D}G(X^k)$ is surjective for all $k$, we have by (\ref{eq:tol1}) and (\ref{ieq:vg}),
\[
\begin{array}{rcl}
& &\| \widehat{\Delta X}^k \| = \|(\mathrm{D}G(X^k))^*[ \Delta Z^k ]\| \\[2mm]
&\leq& \normmm{\mathrm{D}G(X^k))^* \circ \big(\mathrm{D}J(X^k)\big)^{-1}} \cdot \|J(X^k)[ \Delta Z^k ] \|_F \\[2mm]
&=& \normmm{(\mathrm{D}G(X^k))^* \circ \big(J(X^k)\big)^{-1}}\cdot \| V(X^k) - G(X^k) \|_F \\[2mm]
&\leq& \normmm{(\mathrm{D}G(X^k))^* \circ \big(J(X^k)\big)^{-1}}\cdot \big(\|V(X^k)\|_F + \| G(X^k) \|_F \big) \\[2mm]
&\le& (1+\overline{\eta}_{k})  \normmm{(\mathrm{D}G(X^k))^* \circ \big(J(X^k)\big)^{-1}} \cdot \| G(X^k)\|_{F} \\[2mm]
&\leq& (1+\overline{\eta}_{k}) \normmm{(\mathrm{D}G(X^k))^* \circ \big( \mathrm{D}G(X^k) \circ (\mathrm{D}G(X^k))^* \big)^{-1}} \cdot \| G(X^k)\|_{F}\\[2mm]
&=& (1+\overline{\eta}_{k}) \normmm{(\mathrm{D}G(X^k))^\dag} \cdot \| G(X^k) \|_F.
\end{array}
\]
\end{proof}

On the upper bound of the iterate $\widehat{\eta}_k$ generated by Algorithm \ref{nm1}, we have the following result.
\begin{lemma}\label{lemma:residual}
Assume that $\mathrm{D}G(X^k):T_{X^k}(\Rnn \times\co(n)\times \mathcal{V})\to T_{G(X^k)}\Rnn$ is surjective for all $k$.
If the linear matrix equation {\rm (\ref{eq:le})} is solvable  such that conditions {\rm (\ref{eq:tol1})}
and {\rm (\ref{eq:tol2})}  are satisfied for all $k$, then one has for all $k$,
\BE\label{eq:resnorm}
\widehat{\eta}_k \leq \min\left\{\frac{\overline{\sigma}_k}{\lambda_{\min}\big(\mathrm{D}G(X^k) \circ (\mathrm{D}G(X^k))^*\big)+\overline{\sigma}_k} +  \overline{\eta}_k, \widehat{\eta}_{\max} \right\},
\EE
where $\la_{\min}(\cdot)$ means the smallest eigenvalue of a positive definite linear operator.
\end{lemma}

\begin{proof}
Let $\mathrm{id}$, $J(X^k)$ and $V(X^k)$ be defined as in Lemma \ref{lemma:direnorm}.  By assumption, $\mathrm{D}G(X^k)$ is surjective for all $k$.
It follows from (\ref{eq:direstep}) and (\ref{ieq:vg}) that for all $k$,
\[
\begin{array}{rcl}
& & \| G(X^k) + \mathrm{D} G(X^k)[ \widehat{\Delta X}^k ] \|_F\\[2mm]
&=& \|G(X^k) + \mathrm{D} G(X^k) \big[ (\mathrm{D}G(X^k))^*[ \Delta Z_k ] \big]\|_F\\[2mm]
&=& \|G(X^k)+ \big(\mathrm{D} G(X^k) \circ (\mathrm{D}G(X^k))^*\big)\circ\big(J(X^k)\big)^{-1}
[V(X^k) - G(X^k) ]\|_F \\[2mm]
&\leq& \normmm{\mathrm{id} - \big(\mathrm{D} G(X^k) \circ (\mathrm{D}G(X^k))^*\big)\circ\big(J(X^k) \big)^{-1}}\cdot \|G(X^k)\|_F\\[2mm]
& & + \normmm{\big(\mathrm{D} G(X^k) \circ (\mathrm{D}G(X^k))^*\big)\circ\big(J(X^k) \big)^{-1}}\cdot \|V(X^k)\|_F\\[2mm]
&\leq& \displaystyle \Big(\frac{\overline{\sigma}_k}{\lambda_{\min}\big(\mathrm{D}G(X^k) \circ (\mathrm{D}G(X^k))^*\big)+\overline{\sigma}_k} +  \overline{\eta}_k \Big)\|G(X^k)\|_F.
\end{array}
\]
This, together with (\ref{eq:tol2}), yields (\ref{eq:resnorm}).
\end{proof}

On the repeat-loop of Algorithm \ref{nm1}, we have the following lemma.
\begin{lemma}\label{lemma:finite}
Assume that in the $k$-th iteration of Algorithm {\rm \ref{nm1}}, the operator $\mathrm{D}G(X^k):T_{X^k}(\Rnn \times\co(n)\times \mathcal{V})\to T_{G(X^k)}\Rnn$ is surjective and the linear matrix equation {\rm (\ref{eq:le})} is solvable  such that conditions {\rm (\ref{eq:tol1})} and {\rm (\ref{eq:tol2})}  are satisfied.  Then the repeat-loop terminates
in finite steps with $\Delta X^k$ and $\eta_k$ satisfying
\BE\label{eq:bade}
\left\{
\begin{array}{l}
\| G(X^k) + \mathrm{D} G(X^k)[\Delta X^k] \|_F \leq \eta_k \| G(X^k) \|_F,\\[2mm]
\| G(X^{k+1}) \|_F \leq \big( 1 - t(1-\eta_k) \big)\|G(X^k)\|_F.
\end{array}
\right.
\EE
\end{lemma}

\begin{proof}
In the repeat-loop, the search direction $\Delta X^k$ is scaled by some $\theta_j \in [\theta_{\min},\theta_{\max}]$
at the $j$-th step. Hence, at the $m$-th step of the repeat-loop, we get
\[
\Delta X^k = \prod^m_{j=1} \theta_j \widehat{\Delta X}^k \quad \text{and} \quad \eta_k = 1-\prod^m_{j=1} \theta_j(1-\widehat{\eta}_k).
\]
Also, we have
\[
\Theta_{m} :=\prod^m_{j=1} \theta_j \leq \prod^m_{j=1} \theta_{\max} = \theta^m_{\max}.
\]

We note that $G$ is continuously differentiable and $0<\theta_{\max}<1$. According to (\ref{def:pb}) and (\ref{s1:e2}), we obtain for all $m$ sufficiently large,
\[
\|G\big(R_{X^k}(\Theta_{m} \widehat{\Delta X}^k)\big) - G(X^k) - \mathrm{D} G(X^k) [ \Theta_{m} \widehat{\Delta X}^k ] \|_F \leq
\epsilon_k\| \Theta_{m} \widehat{\Delta X}^k \|,
\]
and then
\BE\label{eq:lips22}
\|\widehat{G}_{X^k}(\Theta_{m} \widehat{\Delta X}^k)-\widehat{G}_{X^k}(0_{X^k})- \mathrm{D}\widehat{G}_{X^k}(0_{X^k})[\Theta_{m} \widehat{\Delta X}^k] \|_F \leq \epsilon_k\| \Theta_{m} \widehat{\Delta X}^k \|,
\EE
where $\epsilon_k:= ((1-t)(1-\widehat{\eta}_k))/( (1+\overline{\eta}_{\max})\normmm{ (\mathrm{D} F(X^k))^\dag})$.

We now show that the repeat-loop terminates in finite steps. Let $\widehat{m}$ be the smallest integer such that (\ref{eq:lips22}) holds.
Let $\Delta X^k := \Theta_{\widehat{m}} \widehat{\Delta X}^k$.
We get by (\ref{eq:direstep}), (\ref{def:pb}), and (\ref{s1:e2}),
\[
\begin{array}{rcl}
& &\|G(X^k) + \mathrm{D} G(X^k) [\Delta X^k] \|_F \\[2mm]
&=&\|\widehat{G}_{X^k}(0_{X^k})  + \mathrm{D} \widehat{G}_{X^k}(0_{X^k}) [\Delta X^k] \|_F \\[2mm]
&=&\|(1-\Theta_{\widehat{m}})\widehat{G}_{X^k}(0_{X^k}) +  \Theta_{\widehat{m}}\widehat{G}_{X^k}(0_{X^k})
+ \Theta_{\widehat{m}} \mathrm{D} \widehat{G}_{X^k}(0_{X^k}) [\widehat{\Delta X}^k] \|_F \\[2mm]
&\leq & (1-\Theta_{\widehat{m}}) \|\widehat{G}_{X^k}(0_{X^k})\|_F + \Theta_{\widehat{m}}\| \widehat{G}_{X^k}(0_{X^k})
+ \mathrm{D} \widehat{G}_{X^k}(0_{X^k}) [\widehat{\Delta X}^k] \|_F \\[2mm]
&=& (1-\Theta_{\widehat{m}}) \|\widehat{G}_{X^k}(0_{X^k})\|_F + \Theta_{\widehat{m}}
\widehat{\eta}_k \| \widehat{G}_{X^k}(0_{X^k}) \|_F \\[2mm]
&=&\big(1-\Theta_{\widehat{m}} + \Theta_{\widehat{m}}\widehat{\eta}_k \big) \| \widehat{G}_{X^k}(0_{X^k}) \|_F \\[2mm]
&=&\big(1-\Theta_{\widehat{m}}(1-\widehat{\eta}_k)\big) \| G(X^k) \|_F \\[2mm]
&=& \eta_k \| G(X^k) \|_F.
\end{array}
\]
This, together with Lemmas \ref{lemma:direnorm} and \ref{lemma:residual}, (\ref{eq:direstep}), and (\ref{eq:lips22}), yields
\begin{eqnarray*}
& &\| G(X^{k+1}) \|_F = \| \widehat{G}_{X^k}(\Delta X^k) \|_F \\
&\leq& \| \widehat{G}_{X^k}(0_{X^k}) + \mathrm{D} \widehat{G}_{X^k}(0_{X^k}) [\Delta X^k] \|_F\\
&&\qquad+ \|  \widehat{G}_{X^k}(\Delta  X^k) - \widehat{G}_{X^k}(0_{X^k}) - \mathrm{D} \widehat{G}_{X^k}(0_{X^k}) [\Delta X^k]\|_F \\
&\leq&  \eta_k \|\widehat{G}_{X^k}(0_{X^k}) \|_F +\epsilon_k \Theta_{\widehat{m}}\| \widehat{\Delta X}^k  \|_F \\
&\leq&  \eta_k \|\widehat{G}_{X^k}(0_{X^k}) \|_F +\epsilon_k\Theta_{\widehat{m}}(1+\overline{\eta}_k)\normmm{(\mathrm{D} G(X^k))^\dag}\cdot\|G(X^k)\|_F \\
&\leq&  \eta_k \|\widehat{G}_{X^k}(0_{X^k}) \|_F +\epsilon_k\Theta_{\widehat{m}}(1+\overline{\eta}_{\max})\normmm{(\mathrm{D} G(X^k))^\dag}\cdot\|G(X^k)\|_F \\
&=&  \big( \eta_k  +\epsilon_k\Theta_{\widehat{m}}(1+ \overline{\eta}_{\max})\normmm{(\mathrm{D} G(X^k))^\dag}\big)\|G(X^k)\|_F \\
&=&  \Big( \eta_k  +\Theta_{\widehat{m}} \frac{(1-t)(1-\widehat{\eta}_k)}{(1+\overline{\eta}_{\max}) \normmm{(\mathrm{D} G(X^k))^\dag}}
(1+ \overline{\eta}_{\max}) \normmm{(\mathrm{D} G(X^k))^\dag}\Big)\|G(X^k)\|_F \\
&=& \big( \eta_k + \Theta_{\widehat{m}} (1-t)(1-\widehat{\eta}_k) \big) \|G(X^k)\|_F\\
&=& \Big( \eta_k + 1 - \big(1 - \Theta_{\widehat{m}}(1-\widehat{\eta}_k)\big) -t +
t\big(1 - \Theta_{\widehat{m}}(1-\widehat{\eta}_k)\big) \Big) \|G(X^k)\|_F \\
&=& \big( \eta_k + 1 - \eta_k - t + t\eta_k \big) \|G(X^k)\|_F \\
&=& \big( 1- t(1-\eta_k) \big) \|G(X^k)\|_F.
\end{eqnarray*}
\end{proof}

We now establish the global convergence of Algorithm \ref{nm1}. We need the following assumption.
\begin{assumption} \label{ass:nons}
$\mathrm{D} G(\overline{X}):T_{\overline{X}}(\Rnn \times\co(n)\times \mathcal{V})\to T_{G(\overline{X})}\Rnn$ is surjective,
where $\overline{X}\in\Rnn \times\co(n)\times \mathcal{V}$ is an accumulation point of the sequence $\{X^k\}$ generated by Algorithm {\rm \ref{nm1}}.
\end{assumption}

We have the following theorem on the global convergence of Algorithm \ref{nm1}.
\begin{theorem}\label{theorem:global}
Let $\overline{X}$ be an accumulation point of the sequence $\{X^k\}$ generated by Algorithm {\rm \ref{nm1}}. Suppose that Assumption {\rm \ref{ass:nons}} is satisfied.
Then the whole sequence $\{X^k\}$ converges to $\overline{X}$ and $G(\overline{X})={\bf 0}_{n\times n}$.
\end{theorem}

\begin{proof}
By assumption, $\mathrm{D}G(\overline{X})$ is surjective. In addition, $G$ is continuously differentiable.
Thus there exists a sufficiently small constant $\overline{\delta} >0$  such that for any $X$ in a ball $B_{\overline{\delta}}(\overline{X})$ of $\overline{X}$, the liner operator $\mathrm{D} G(X)$ is surjective and
\BE\label{eq:2D}
\normmm{( \mathrm{D}G(X))^\dag} \leq 2 \normmm{( \mathrm{D}G(\overline{X}))^\dag}.
\EE
Let $\epsilon:=  ((1-t)(1-\widehat{\eta}_{\max}))/( 2((1+\overline{\eta}_{\max})\normmm{\mathrm{D}G(\overline{X}))^\dag})$. Then there exist two constants $\delta_1>0$ and $\mu_1 >0$ such that
\BE\label{eq:lip}
\| \widehat{G}_{X}(\Delta X) - \widehat{G}_X(0_{X}) - \mathrm{D} \widehat{G}_{X}(0_{X}) [\Delta X] \|_F \leq \epsilon \| \Delta X\|
\EE
for all $X\in B_{\delta_1}(\overline{X})$ and $\|\Delta X\| \leq \mu_1$. Let $\delta = \min \{\overline{\delta},\delta_1\}$.
Since $\overline{X}$ is an accumulation point of the sequence $\{X^k\}$, there exist infinitely many $k$ such that
$X^k \in B_{\delta}(\overline{X})$. Let $\widehat{m}$ be the smallest integer such that
\[
2\theta^{\widehat{m}}_{\max} (1+ \overline{\eta}_{\max}) \normmm{(\mathrm{D}G(\overline{X}))^\dag} \cdot\| G(X^0)\|_F <  \mu_1.
\]
Let $\Theta_{\widehat{m}} := \prod^{\widehat{m}}_{i=1} \theta_i$.
By Lemma \ref{lemma:direnorm}, (\ref{eq:2D}), and the above inequality, we have
\begin{equation}\label{eq:ssp}
\begin{array}{rcl}
\| \Theta_{\widehat{m}} \widehat{\Delta X}^k \| &\leq& \theta^{\widehat{m}}_{\max} \| \widehat{\Delta X}^k \| \\[2mm]
&\leq& \theta^{\widehat{m}}_{\max} (1+ \overline{\eta}_{k}) \normmm{(\mathrm{D} G(X^k))^\dag}\cdot\| G(X^k)\|_{F} \\[2mm]
&\leq& 2\theta^{\widehat{m}}_{\max} (1+ \overline{\eta}_{\max}) \normmm{(\mathrm{D}G(\overline{X}))^\dag}\cdot\| G(X^0)\|_F \\[2mm]
&<&  \mu_1
\end{array}
\end{equation}
for $X^k\in B_\delta(\overline{X})$. This, together with (\ref{eq:lip}), gives rise to
\[
\| \widehat{G}_{X^k}(\Theta_{\widehat{m}} \widehat{\Delta X}^k) - \widehat{G}_{X^k}(0_{X^k})
-\mathrm{D} \widehat{G}_{X^k}(0_{X^k})[\Theta_{\widehat{m}} \widehat{\Delta X}^k] \|_F
\leq \epsilon \| \Theta_{\widehat{m}} \widehat{\Delta X}^k\|
\]
for $X^k\in B_\delta(\overline{X})$.
By Lemma \ref{lemma:finite}, suppose that the repeat-loop terminates in at most $\widehat{m}$ steps.
Then, for any $X^k \in B_{\delta}(\overline{X})$, we have
\BE\label{bd:etak}
1-\eta_k = \Theta_{\widehat{m}}(1-\widehat{\eta}_k) \geq \theta^{\widehat{m}}_{\min}(1-\widehat{\eta}_{\max}) > 0.
\EE
Since $\overline{X}$ is an accumulation point of $\{X^k\}$, there exists a subsequence $\{X^{k_j}\} \in B_{\delta}(\overline{X})$. Then
\[
\sum_{k\geq 0}(1-\eta_k) = \sum_{k\neq k_j}(1-\eta_k) + \sum_{j\geq 0}(1-\eta_{k_j})
\ge \sum_{k\neq k_j}(1-\eta_k) + \sum_{j\geq 0}\theta^{\widehat{m}}_{\min}(1-\widehat{\eta}_{\max})
=\infty.
\]
Hence, we obtain by (\ref{eq:bade}),
\[
\begin{array}{rcl}
\|G(X^{k+1})\|_F &\leq& \big( 1-t(1-\eta_k) \big) \| G(X^{k}) \|_F
\le \|G(X^0)\|_F \prod_{0\leq l \le k} \big( 1-t(1-\eta_l) \big) \\[2mm]
&\leq& \|G(X^0)\|_F \exp \Big ( -t\sum_{0\leq l \le k}(1-\eta_l) \Big)\to 0,\quad k\to\infty.
\end{array}
\]
Thus we have
\begin{eqnarray}\label{bd:gk}
\lim\limits_{k\to\infty}\|G(X^{k+1})\|_F = 0.
\end{eqnarray}
Next, we show that $\{X^k\}$ converges to $\overline{X}$. We get by (\ref{eq:ssp}) and (\ref{bd:etak}), for $X^k \in B_{\delta}(\overline{X})$,
\begin{equation}\label{eq:unibound}
\begin{array}{rcl}
\| \Delta X^k \| &\leq& \theta^{\widehat{m}}_{\max} \| \widehat{\Delta X}^k \| \\[2mm]
&\leq& 2\theta^{\widehat{m}}_{\max} (1+\overline{\eta}_{\max}) \normmm{(\mathrm{D}G(\overline{X}))^\dag}\cdot\|G(X^k)\|_F \\[2mm]
&\leq& \displaystyle \frac{ 2\theta^{\widehat{m}}_{\max}(1+\overline{\eta}_{\max})\normmm{(\mathrm{D}G(\overline{X}))^\dag}}
{\theta^{\widehat{m}}_{\min}(1-\widehat{\eta}_{\max})}
(1-\eta_k) \|G(X^k)\|_F.
\end{array}
\end{equation}
Based on (\ref{bd:gk}) and (\ref{eq:unibound}), we can obtain
\begin{eqnarray}\label{eq:vxc}
\lim\limits_{k\to\infty}\|\Delta X^k \| = 0.
\end{eqnarray}
Thus for all $k$ sufficiently large with $X^k \in B_{\delta}(\overline{X})$, it holds
\BE\label{retr:bd2}
\| \Delta X^k  \|\leq \mu_{\nu},
\EE
where $\mu_{\nu}$ is the constant given in (\ref{retr:bd-2}).
By contradiction, suppose that the sequence $\{X^k\}$ does not converge to $\overline{X}$.
Then there exist infinitely many $k$ such that $X^k \not\in  B_{\delta}(\overline{X})$.
Since $\overline{X}$ is an accumulation point of $\{X^k\}$, there exist two index sets
 $\{m_j\}$ and $\{n_j\}$ such that $\lim_{j\to \infty}X^{m_j} = \overline{X}$, and for each $j$,
\[
\left\{
\begin{array}{rcl}
X^{m_j} &\in& B_{\delta}(\overline{X}), \quad X^{m_j+i} \in B_{\delta}(\overline{X}), \quad i=0,\ldots,n_j-1, \\[2mm]
X^{m_j+n_j} &\not\in& B_{\delta}(\overline{X}), \quad m_j + n_j  < m_{j+1}.
\end{array}
\right.
\]
Thus, we have by  (\ref{retr:bd-2}), (\ref{eq:bade}), (\ref{eq:unibound}), and (\ref{retr:bd2}),
\[
\begin{array}{rcl}
\displaystyle \frac{\delta}{2} &\leq& \mbox{dist}(X^{m_j+n_j}, X^{m_j})
\leq \sum^{m_j+n_j-1}_{k=m_j} \mbox{dist}(X^{k+1}, X^{k})  \\[2mm]
&=&  \displaystyle \sum^{m_j+n_j-1}_{k=m_j} \mbox{dist}\big(R_{X^{k}}(\Delta X^k), X^{k} \big)
\leq \sum^{m_j+n_j-1}_{k=m_j} \nu \| \Delta X^k \| \\[2mm]
&\leq& \displaystyle \sum^{m_j+n_j-1}_{k=m_j} \nu \frac{2\theta^{\widehat{m}}_{\max}(1+\overline{\eta}_{\max}) \normmm{(\mathrm{D}G(\overline{X}))^\dag}}
{\theta^{\widehat{m}}_{\min}(1-\widehat{\eta}_{\max}) }(1-\eta_k) \|G(X^k)\|_F \\[2mm]
&\leq& \displaystyle \sum^{m_j+n_j-1}_{k=m_j}  \frac{2\nu \theta^{\widehat{m}}_{\max}(1+\overline{\eta}_{\max}) \normmm{(\mathrm{D}G(\overline{X}))^\dag}}
{\theta^{\widehat{m}}_{\min}(1-\widehat{\eta}_{\max}) } \times \frac{ \|G(X^k)\|_F - \|G(X^{k+1})\|_F }{t} \\[2mm]
&=&  \displaystyle \frac{2\nu \theta^{\widehat{m}}_{\max}(1+\overline{\eta}_{\max}) \normmm{(\mathrm{D}G(\overline{X}))^\dag}}
{t\theta^{\widehat{m}}_{\min}(1-\widehat{\eta}_{\max}) } \big( \|G(X^{m_j})\|_F - \|G(X^{m_j+n_j})\|_F \big) \\[2mm]
&\leq& \displaystyle \frac{2\nu \theta^{\widehat{m}}_{\max}(1+\overline{\eta}_{\max}) \normmm{(\mathrm{D}G(\overline{X}))^\dag}}
{t\theta^{\widehat{m}}_{\min}(1-\widehat{\eta}_{\max}) } \big( \|G(X^{m_j})\|_F - \|G(X^{m_{j+1}})\|_F \big)
\\[2mm]
&\to& 0,\quad \mbox{as }\; j\to\infty,
\end{array}
\]
since $X^{m_j} \to \overline{X}$ as $j\to\infty$.
This is a contradiction. Therefore, $\{X^k\}$ converges to $\overline{X}$.
\end{proof}

\subsection{Quadratic convergence}
In this section, we show the quadratic convergence of Algorithm \ref{nm1}.
First, we have the following result on the backtracking line search procedure.
\begin{lemma}\label{lem:stepsize}
Let $\overline{X}$ be an accumulation point of the sequence $\{X^k\}$ generated by Algorithm {\rm \ref{nm1}}.
Suppose that Assumption {\rm \ref{ass:nons}} is satisfied. Then $\eta_k = \widehat{\eta}_k$ and
$\Delta X^k = \widehat{\Delta X}^k$ for all $k$ sufficiently large.
\end{lemma}

\begin{proof}
We note that $G$ is continuously differentiable.
By assumption, $\mathrm{D}G(\overline{X})$ is surjective. By Theorem \ref{theorem:global}, the sequence $\{X^k\}$
converges to $\overline{X}$ with $G(\overline{X})={\bf 0}_{n\times n}$. Based on (\ref{eq:2D}),
$\mathrm{D} G(X^k)$ is surjective and satisfies
\[
\normmm{( \mathrm{D}G(X^k))^\dag} \leq 2 \normmm{( \mathrm{D}G(\overline{X}))^\dag}
\]
for all $k$ sufficiently large. By Lemma \ref{lemma:direnorm} and the definition of $\overline{\eta}_k$ in Algorithm {\rm \ref{nm1}},
one has for all $k$ sufficiently large,
\[
\begin{array}{rcl}
\| \widehat{\Delta X}^k \| &\le& (1+\overline{\eta}_k) \normmm{(\mathrm{D}G(X^k))^\dag}\cdot\|G(X^k)\|_F \\[2mm]
&\le& (1+\overline{\eta}_{\max}) \normmm{(\mathrm{D}G(X^k))^\dag}\cdot\|G(X^k)\|_F  \\[2mm]
&\leq & 2(1+\overline{\eta}_{\max}) \normmm{(\mathrm{D}G(\overline{X}))^\dag}\cdot\|G(X^k)\|_F.
\end{array}
\]
Based on (\ref{bd:gk}) and the above inequality, we can obtain
$\lim_{k\to\infty}\| \widehat{\Delta X}^k \| = 0$.
Hence, for all $k$ sufficiently large, it holds that
\[
\| \widehat{G}_{X^k}(\widehat{\Delta X}^k) - \widehat{G}_{X^k}(0_{X^k}) - \mathrm{D} \widehat{G}_{X^k}(0_{X^k}) [\widehat{\Delta X}^k] \|_F
\leq \epsilon  \| \widehat{\Delta X}^k\| \leq \epsilon_k  \| \widehat{\Delta X}^k\|,
\]
where the condition $\epsilon \le  \epsilon_k$ is used with $\epsilon_k$ and $\epsilon$ being defined in (\ref{eq:lips22}) and (\ref{eq:lip}). Based on the analysis in Lemma \ref{lemma:finite}, this implies that $\eta_k = \widehat{\eta}_k$ and $\Delta X^k = \widehat{\Delta X}^k$ for all $k$ sufficiently large.
\end{proof}

We now establish the quadratic convergence of Algorithm \ref{nm1}.
\begin{theorem} \label{th:qc}
Let $\overline{X}$ be an accumulation point of the sequence $\{X^k\}$ generated by Algorithm {\rm \ref{nm1}}. Suppose that Assumption {\rm \ref{ass:nons}} is satisfied.
Then the whole sequence $\{X^k\}$ converges to $\overline{X}$ quadratically.
\end{theorem}

\begin{proof}
By Theorem \ref{theorem:global} and Lemma \ref{lem:stepsize}, $\{X^k\}$ converges to $\overline{X}$ with $G(\overline{X})={\bf 0}_{n\times n}$ and $\eta_k = \widehat{\eta}_k$ and $\Delta X^k = \widehat{\Delta X}^k$ for all $k$ sufficiently large with $\|\Delta X^k\|=\|\widehat{\Delta X}^k\|\to 0$
as $k\to\infty$. We note that $G$ is continuously differentiable and, by assumption, $\mathrm{D}G(\overline{X})$ is surjective.
Based on (\ref{eq:2D}), $\mathrm{D} G(X^k)$ is surjective, and there exists a constant $\overline{\la}_{\min} >0$ such that
\BE\label{bd:gxk-eig}
\normmm{( \mathrm{D}G(X^k))^\dag} \leq 2 \normmm{( \mathrm{D}G(\overline{X}))^\dag}\;\mbox{and}\; \lambda_{\min}(\mathrm{D}G(X^k) \circ (\mathrm{D}G(X^k))^*)\ge \overline{\la}_{\min} >0
\EE
for all $k$ sufficiently large. Moreover, there exist two constants $L_1,L_2>0$ such that for all $k$ sufficiently large,
\BE\label{eq:lip4}
\left\{
\begin{array}{l}
\| G(X^k) - G(\overline{X}) \|_F \leq L_1 \mbox{dist}(X^k, \overline{X}), \\[2mm]
\| \widehat{G}_{X^k}(\Delta X^k) - \widehat{G}_{X^k}(0_{X^k}) - \mathrm{D} \widehat{G}_{X^k}(0_{X^k}) [\Delta X^k] \|_F
\leq L_2 \| \Delta X^k\|^2,\\[2mm]
\mbox{dist}\big(X^k, R_{X^k}(\Delta X^k) \big)\le\nu\|\Delta X^k\|,
\end{array}
\right.
\EE
where $\nu$ is the constant given in (\ref{retr:bd-2}).

We obtain by (\ref{eq:resnorm}), (\ref{bd:gxk-eig}), (\ref{eq:lip4}), and the definition of $\overline{\sigma}_k$
and $\overline{\eta}_k$ in Algorithm {\rm \ref{nm1}} for all $k$ sufficiently large,
\begin{equation}\label{bd:etak-2}
\begin{array}{rcl}
\widehat{\eta}_k &\leq&  \displaystyle \frac{\overline{\sigma}_k}{\lambda_{\min}\big(\mathrm{D}G(X^k) \circ (\mathrm{D}G(X^k))^*\big)+\overline{\sigma}_k} +\overline{\eta}_k \\[2mm]
&\leq& \displaystyle \frac{1}{\overline{\la}_{\min}+\overline{\sigma}_k}\overline{\sigma}_k + \overline{\eta}_k
\le \displaystyle \frac{1}{\overline{\la}_{\min}}\|G(X^{k})\|_F + \|G(X^{k})\|_F \\[2mm]
&\leq& \displaystyle \frac{1+\overline{\la}_{\min}}{\overline{\la}_{\min}}L_1\mbox{dist}(X^k, \overline{X})
\equiv c_1\mbox{dist}(X^k, \overline{X}),
\end{array}
\end{equation}
where $c_1 := (L_1(1+\overline{\la}_{\min}))/ \overline{\la}_{\min}$.
By Lemmas \ref{lemma:direnorm} and \ref{lemma:residual},  (\ref{eq:direstep}), (\ref{bd:gxk-eig}), (\ref{eq:lip4}), and  (\ref{bd:etak-2}),
we have for all $k$ sufficiently large,
\begin{equation}\label{bd:gxk1}
\begin{array}{rcl}
& &\| G(X^{k+1}) \|_F \\[2mm]
&=& \| G(X^{k+1}) - G(X^k) - \mathrm{D} G(X^k) [\Delta X^k] + G(X^k) + \mathrm{D} G(X^k) [\Delta X^k] \|_F \\[2mm]
&\leq& \| \widehat{G}_{X^k}(\Delta X^k) - \widehat{G}_{X^k}(0_{X^k}) - \mathrm{D} \widehat{G}_{X^k}(0_{X^k}) [\Delta X^k] \|_F\\[2mm]
&&\qquad+ \|\widehat{G}_{X^k}(0_{X^k}) + \mathrm{D} \widehat{G}_{X^k}(0_{X^k}) [\Delta X^k]\|_F \\[2mm]
&\leq& L_2 \|\Delta X^k\|^2 + \widehat{\eta}_k \|G(X^k)\|_F\\[2mm]
&\leq& L_2\big( (1+\overline{\eta}_k) \normmm{ (\mathrm{D} G(X^k))^\dag }\cdot\|G(X^k)\| \big)^2  + \widehat{\eta}_k \|G(X^k)\|_F \\[2mm]
&\leq& L_2\big( (1+\overline{\eta}_k) \normmm{ (\mathrm{D} G(X^k))^\dag } \big)^2 \big( L_1\mbox{dist}(X^k, \overline{X}) \big)^2
+ \widehat{\eta}_k L_1 \mbox{dist}(X^k, \overline{X}) \\[2mm]
&\leq& L_2\big( 2(1+\overline{\eta}_{\max}) L_1\normmm{(\mathrm{D}G(\overline{X}))^\dag} \big)^2 \big( \mbox{dist}(X^k, \overline{X})\big)^2
+ c_1 L_1 \big(\mbox{dist}(X^k, \overline{X}) \big)^2 \\[2mm]
& \equiv & c_2 \big(\mbox{dist}(X^k, \overline{X})\big)^2,
\end{array}
\end{equation}
where $c_2:=L_2\big( 2(1+\overline{\eta}_{\max}) L_1\normmm{(\mathrm{D}G(\overline{X}))^\dag } \big)^2+ c_1 L_1$.
We have by  Lemma \ref{lemma:direnorm}, (\ref{bd:gxk-eig}),  (\ref{eq:lip4}), and (\ref{bd:gxk1}), for all $k$ sufficiently large,
\[
\begin{array}{rcl}
\mbox{dist}(X^{k+1}, \overline{X})
&\leq& \displaystyle \sum^{\infty}_{j=k+1} \mbox{dist} (X^j, X^{j+1})
= \sum^{\infty}_{j=k+1} \mbox{dist} \big(X^j, R_{X^j}( \Delta X^j) \big)\\[2mm]
&\leq&  \displaystyle \sum^{\infty}_{j=k+1} \nu \| \Delta X^j \|
\leq \sum^{\infty}_{j=k+1} 2\nu  (1+\overline{\eta}_{\max}) \normmm{ (\mathrm{D}G(\overline{X}))^\dag }\cdot\|G(X^j)\|_F \\[2mm]
&\leq&  \displaystyle 2\nu (1+\overline{\eta}_{\max}) \normmm{ (\mathrm{D}G(\overline{X}))^\dag }
\sum^{\infty}_{j=0}\big(1-t(1-\widehat{\eta}_{\max}) \big)^j \|G(X^{k+1})\|_F \\[2mm]
&=& \displaystyle \frac{2\nu (1+\overline{\eta}_{\max}) \normmm{ (\mathrm{D}G(\overline{X}))^\dag }}
{t(1-\widehat{\eta}_{\max})} \|G(X^{k+1})\|_F \\[2mm]
&\equiv& \displaystyle c_3 \big(\mbox{dist}(X^k, \overline{X})\big)^2,
\end{array}
\]
where $c_3=:(2c_2\nu (1+\overline{\eta}_{\max}) \normmm{ (\mathrm{D}G(\overline{X}))^\dag })/(t(1-\widehat{\eta}_{\max}))$.
Thus the proof is complete.

\end{proof}

\subsection{Surjectivity conditions of $\mathrm{D}G(\cdot)$}

We have the following result on the surjectivity of $\mathrm{D}G(\overline{X})$, where $\overline{X}$ is an accumulation point of
the sequence $\{X^k\}$ generated by Algorithm \ref{nm1}.
\begin{theorem}\label{riediff:frk2}
Let $\overline{X}:=(\overline{S},\overline{Q},\overline{V})\in\Rnn \times \co(n) \times \cv$ be an accumulation point of the sequence $\{X^k:=(S^k,Q^k,V^k)\}$ generated by Algorithm {\rm \ref{nm1}}. Then $\mathrm{D} G(\overline{X})$ is surjective if and only if
\BE\label{sur:cond}
\mathrm{null}\left(
\left[
\begin{array}{c}
\mathrm{Diag}\big(\mathrm{vec}(\overline{S})\big) \\[2mm]
(I_{n^2} - \widehat{P})\big( (\overline{S}\odot \overline{S}) \otimes I_{n}-  I_{n}\otimes(\overline{S}\odot \overline{S})^T \big) \\[2mm]
\mathrm{Diag}\big(\mathrm{vec}(W)\big) (\overline{Q} \otimes \overline{Q})^T
\end{array}
\right] \right)
= \{ {\bf 0}_{n^2} \},
\EE
where $W\in\Rnn$ is defined in Appendix A and $\widehat{P}\in \R^{n^2\times n^2}$ is the vectorized transpose matrix such that
\[
{\rm vec}(A^T) = \widehat{P}\,{\rm vec}(A), \quad \forall A\in \R^{n\times n}.
\]

\end{theorem}
\begin{proof}
Notice
$T_{G(\overline{X})}\Rnn=\im(\mathrm{D}G(\overline{X}))\oplus \im(\mathrm{D}G(\overline{X}))^\perp$ and
$\im(\mathrm{D}G(\overline{X}))^\perp=\ker\big( (\mathrm{D}G(\overline{X}))^*\big)$,
where $\im(\mathrm{D}G(\overline{X}))$ and $\ker\big( (\mathrm{D}G(\overline{X}))^*\big)$ denote the image of $\mathrm{D}G(\overline{X})$ and the kernel of $(\mathrm{D}G(\overline{X}))^*$, respectively.
Then the linear operator $\mathrm{D}G(\overline{X})$ is surjective if and only if $\ker( (\mathrm{D}G(\overline{X}))^*)=\{\mathbf{0}_{n\times n}\}$.

We now derive a sufficient and necessary condition for $\ker\big( (\mathrm{D}G(\overline{X}))^*\big)=\{\mathbf{0}_{n\times n}\}$. Let $\Delta Z\in T_{G(\overline{X})}\Rnn$ be such that $(\mathrm{D}G(\overline{X}))^*[\Delta Z]=0_{\overline{X}}$.
We have by the expression of $\mathrm{D}G(\cdot)^*$ given in Appendix A, $\ker\big( (\mathrm{D}G(\overline{X}))^*\big)=\{\mathbf{0}_{n\times n}\}$ if and only if the following equation
\[
\left\{
\begin{array}{l}
\overline{S}\odot \Delta Z = \mathbf{0}_{n\times n},\\[2mm]
[ \overline{Q}(\Lambda + \overline{V})(\overline{Q})^T, (\Delta Z)^T ]
 + [ \overline{Q}(\Lambda + \overline{V})^T(\overline{Q})^T, \Delta Z ]  = \mathbf{0}_{n\times n}, \\[2mm]
W\odot((\overline{Q})^T\Delta Z\overline{Q}) = \mathbf{0}_{n\times n}
\end{array}
\right.
\]
has only a zero solution $\Delta Z=\mathbf{0}_{n\times n}$ or
\[
\left\{
\begin{array}{l}
\mathrm{Diag}\big(\mathrm{vec}(\overline{S})\big) \mathrm{vec}(\Delta Z) = \mathbf{0}_{n^2},\\[2mm]
(I_{n^2} - \widehat{P})(\overline{Q} \otimes \overline{Q})\big(  (\Lambda+\overline{V})\otimes I_n  -  I_{n} \otimes (\Lambda+\overline{V})^T  \big)( \overline{Q} \otimes \overline{Q})^T \mathrm{vec}(\Delta Z) = \mathbf{0}_{n^2},\\[2mm]
\mathrm{Diag}\big(\mathrm{vec}(W)\big) (\overline{Q} \otimes \overline{Q})^T \mathrm{vec}(\Delta Z) = \mathbf{0}_{n^2}
\end{array}
\right.
\]
has only a zero solution $\mathrm{vec}(\Delta Z)=\mathbf{0}_{n^2}$, where the relation
\[
\widehat{P}(A \otimes B) = (B \otimes A)\widehat{P},\quad  \forall A,B\in \Rnn
\]
is used \cite[p.448]{Bernstein}. This is reduced to (\ref{sur:cond}). The proof is complete.
\end{proof}

\section{Extensions}\label{sec5}
In this section, we extend the proposed Riemannian inexact Newton-CG method to the case of prescribed entries.
The nonnegative inverse eigenvalue problem with prescribed entries can be stated as follows:

{\bf NIEP-PE.} {\em Given a self-conjugate set of $n$ complex numbers $\{\lambda_1, \lambda_2, \ldots, \lambda_n\}$, find an $n$-by-$n$
real nonnegative matrix $C$ such that its eigenvalues are $\lambda_1, \lambda_2, \ldots , \lambda_n$ and
\[
(C)_{ij}=(C_a)_{ij},\quad \forall (i,j)\in\cl,
\]
where $\cl\subset\cn$ is a given index subset and $C_a$ is any  given $n$-by-$n$ nonnegative matrix such that $\{(C_a)_{ij}\;|\; (i,j)\in\cl\}$  are prescribed entries.}

Define  the matrix $\widehat{U} \in \Rnn$   by $(\widehat{U})_{ij} =1$, if $(i,j) \in \cl$; $0$,  otherwise.
Let the  matrix $\widehat{C}_a\in\Rnn$ be defined by $\widehat{C}_a := \widehat{U} \odot C_a$. Also, define a set $\cz$ by
\[
\cz:=\{S\in\Rnn \ | \ \widehat{U} \odot S= \mathbf{0}_{n\times n}\}.
\]
Then the NIEP-PE is to solve the following nonlinear equation:
\BE\label{NIEP-PE}
H(S,Q,V)= \mathbf{0}_{n\times n}
\EE
for $(S,Q,V)\in\cz\times \co(n)\times \cv$,
where    $H:\cz\times \co(n)\times \cv\to\Rnn$ is defined by
\[
H(S,Q,V)= \widehat{C}_a + S \odot S - Q(\Lambda + V)Q^T,\quad (S,Q,V)\in\cz \times \co(n) \times \cv.
\]
Obviously, $H$ is smooth mapping from the product manifold $\cz\times \co(n)\times \cv$ to the linear space $\Rnn$.

We note that the dimension of $\cz\times \co(n)\times \cv$ is given by
\[
\dim (\cz\times \co(n)\times \cv) = n^2 - |\cl| + \frac{n(n-1)}{2} + |\cj|.
\]
We point out that the nonlinear equation $H(S,Q,V)=\mathbf{0}_{n\times n}$ is under-determined over
$\cz\times \co(n)\times \cv$ if the problem size $n$ is large and the number  $|\cl|$ of prescribed entries is small.
We also remark that, if $(\overline{S},\overline{Q},\overline{V})\in\cz\times \co(n)\times \cv$ is a solution to
$H(S,Q,V)= \mathbf{0}_{n\times n}$, then $\overline{C}:=\widehat{C}_a+\overline{S}\odot\overline{S}$
 is a solution to the NIEP-PE.

As in section \ref{sec4}, one may apply Algorithm \ref{nm1} to solving the nonlinear equation  (\ref{NIEP-PE}). Under some mild conditions, the global and quadratic convergence  can be established by a similar way as in section \ref{sec4}.

\section{Numerical Tests}\label{sec6}

In this section, we report the numerical performance of Algorithm \ref{nm1} for solving the NIEP and the NIEP-PE via solving the nonlinear equations (\ref{eq:NIEP1}) and  (\ref{NIEP-PE}).
All the numerical tests are carried out by using {\tt MATLAB} 7.1 running on a workstation with a Intel Xeon CPU E5-2687W at  3.10 GHz and 32 GB of RAM.
To illustrate the efficiency of our algorithm, we compare  Algorithm \ref{nm1} with the alternating projection method \cite{O06}, the Riemannian Fletcher-Reeves conjugate gradient method (RFR) \cite{YBZC16} and the geometric Polak-Ribi\`{e}re-Polyak-based nonlinear conjugate gradient method (GPRP) \cite{ZJB16}. The alternating projection method in \cite{O06} is employed to solve the NIEP:
\BE\label{pro:niep}
\mbox{Find $C\in\cp\cap\Rnn_+$}
\EE
and  the NIEP-PE:
\BE\label{pro:niep-pe}
\mbox{Find $C\in\cp\cap\cq$},
\EE
where
\[
\cp=\{ A \in \Cnn \ | \ \mbox{$A=UTU^H$ for some unitary matrix $U$ and  some $T\in\ct$} \}
\]
and
\[
\cq=\{C\in\Rnn_+\ | \ \mbox{$(C)_{ij}=(C_a)_{ij}$ for all $(i,j)\in\cl$}\}.
\]
Here, $\ct = \{ T \in \Cnn\ | \ \mbox{$T$ is upper triangular with spectrum $\{\lambda_1, \lambda_2, \ldots , \lambda_n\}$}\}$. The associated alternating projection algorithm for solving problem (\ref{pro:niep}) (problem (\ref{pro:niep-pe}), respectively) is stated as follows.

\begin{algorithm} \label{ap}
 {\rm (Alternating projection algorithm)}
\begin{description}
\item [{Step 0.}] Choose an initial point $C^0\in\Rnn_+$ {\rm (}$C^0\in\cq$, respectively{\rm )}. Let $k:=0$.
\item [{Step 1.}] Calculate a Schur decomposition of $C^k=U^kT^k(U^k)^H$.
\item [{Step 2.}] Set $Y^{k+1}=P_\cp(U^k,T^k)$, where $P_\cp(U^k,T^k)$ is defined as in \cite[Definition 4.2]{O06}.
\item [{Step 3.}] Set $C^{k+1}=P_{\Rnn_+}(Y^{k+1})$ {\rm (}$C^{k+1}=P_{\cq}(Y^{k+1})$, respectively{\rm )}, where $P_{\Rnn_+}(Y^{k+1})$ is the  projection of $Y^{k+1}$ onto $\Rnn_+$.
\item [{Step 4.}] Replace $k$ by $k+1$ and go to {\bf Step 1}.
\end{description}
\end{algorithm}

The two Riemannian conjugate gradient methods RFR and GPRP in \cite{YBZC16,ZJB16} are used to solve the following Riemannian optimization problems:
\BE\label{opt:rp}
\begin{array}{ll}
{\min}  &   \displaystyle\phi(S,Q,V):=\frac{1}{2}\|G(S,Q,V)\|_F^2 \\[2mm]
{\rm s.t.} & (S,Q,V)\in \Rnn \times \co(n) \times \cv
\end{array}
\EE
and
\BE\label{opt:rp-pe}
\begin{array}{ll}
{\min}  &   \displaystyle \psi(S,Q,V):=\frac{1}{2}\|H(S,Q,V)\|_F^2\\[2mm]
{\rm s.t.} &  (S,Q,V)\in \cz \times \co(n) \times \cv.
\end{array}
\EE

For Algorithm \ref{nm1} for solving (\ref{eq:NIEP1}), Algorithm \ref{ap} for problem (\ref{pro:niep}), and RFR and GPRP  for problem (\ref{opt:rp}), we randomly generate the starting points by the built-in functions {\tt rand}, {\tt schur}, and  {\tt svd}:
\BE\label{sp1}
\begin{array}{c}
S \odot S= {\tt rand}\,(n,n),  \quad S^0 = S \in\Rnn, \quad C^0 = S^0\odot S^0,  \\[2mm]
\big[ Q^0, V \big] = \mbox{\tt schur}\,(S^0\odot S^0,{\rm 'real'}) , \quad  V^0 = W\odot V.
\end{array}
\EE
For Algorithm \ref{nm1} for solving (\ref{NIEP-PE}), Algorithm \ref{ap}  for problem (\ref{pro:niep-pe}), and RFR and GPRP  for problem (\ref{opt:rp-pe}), the  starting points are generated randomly as follows:
\BE\label{sp1-pe}
\begin{array}{c}
S \odot S= {\tt rand}\,(n,n),  \quad  S^0 = \widehat{U}\odot S \in\cz, \quad C^0 = \widehat{C}_a + S^0\odot S^0,  \\[2mm]
\big[ Q^0, V \big] = \mbox{\tt schur}\,(\widehat{C}_a+S^0\odot S^0,{\rm 'real'}) , \quad  V^0 = W\odot V.
\end{array}
\EE

For comparison purposes, the stopping criteria for Algorithm \ref{nm1}, Algorithm \ref{ap} for problems (\ref{pro:niep}) and (\ref{pro:niep-pe}), and  the two Riemannian conjugate gradient methods in \cite{YBZC16,ZJB16} for problems (\ref{opt:rp}) and (\ref{opt:rp-pe}) are set to be
\[
 \|G(X^k)\|_F < 10^{-8}, \quad  \|H(X^k)\|_F < 10^{-8},\quad \mbox{and} \quad  \|C^k-Y^k\|_F < 10^{-8}.
\]
In our numerical tests, we set $\overline{\sigma}_{\max}=0.01$, $\overline{\eta}_{\max}=0.1$,
$\widehat{\eta}_{\max}=0.9$, $\theta_{\min}=0.1$, $\theta_{\max}=0.9$, and $t=10^{-4}$. The largest number of iterations in Algorithm \ref{ap} is set to be $100000$.
The largest number of outer iterations in Algorithm \ref{nm1} is set to be $100$ and the largest number of iterations in the CG method is set to be $n^2$.

For comparison purposes,  we repeat our experiments over $10$ different starting points. In our numerical tests, `{\tt CT.}', {\tt IT.}', `{\tt NF.}', `{\tt NCG.}',  `{\tt Res.}', and  `{\tt grad.}' mean the averaged total computing time in seconds, the averaged number of iterations, the averaged number of function evaluations, the averaged number of inner CG iterations, the averaged residual $\|G(X^k)\|_F$,  $\|H(X^k)\|_F$, or $\|C^k-Y^k\|_F$,  and the averaged residual $\|\grad \phi(X^k)\|$ or $\|\grad \psi(X^k)\|$ at the final iterates of the corresponding algorithms, accordingly.

\begin{example}\label{ex:1}
We consider the NIEP with varying $n$. Let $\widehat{C}$ be a random $n\times n$ nonnegative matrix with each entry generated from
the uniform distribution on the interval $[0, 1]$. We choose the eigenvalues of $\widehat{C}$ as prescribed spectrum.
\end{example}
\begin{example}\label{ex:2}
We consider the NIEP-PE with varying $n$. Let $\widehat{C}$ be a random $n\times n$ nonnegative matrix with each entry generated from
the uniform distribution on the interval $[0, 1]$. We choose the eigenvalues of $\widehat{C}$ as prescribed spectrum.
Also, we choose the index subset $\cl := \big\{ (i,j) \ | \  0.2 \leq  (\widehat{C})_{ij} \leq 0.3, \; i,j=1,\ldots,n \big \}$.
The nonnegative matrix $C_a\in\Rnn$ with prescribed entries is defined by
$(C_a)_{ij} := (\widehat{C})_{ij}$,  if  $(i,j) \in \cl$; $0$, otherwise.
\end{example}

Tables \ref{table1}--\ref{table2} list the numerical results for Examples \ref{ex:1}--\ref{ex:2}, where ``*" means that the largest number of iterations is reached for some stating points.

We observe from Tables \ref{table1}--\ref{table2} that Algorithm \ref{ap} behaviors better than GPRP and/or  RFR in terms of computing time for small $n$ (e.g., $n=10,20$) while GPRP and RFR work much better than Algorithm \ref{ap} in terms of computing time for $n\ge 50$. However, Algorithm \ref{nm1} is the most effective  in terms of computing time.
\begin{table}[ht]\renewcommand{\arraystretch}{1.2} \addtolength{\tabcolsep}{2pt}
  \caption{Numerical results of Example \ref{ex:1}.}\label{table1}
  \begin{center} {\scriptsize
   \begin{tabular}[c]{|c|c|r|c|c|c|l|l|}
     \hline
Alg.         & $n$ & {\tt CT.} & {\tt IT.} & {\tt NF.} &  {\tt NCG.}   &  {\tt Res.}  &  {\tt grad.}  \\  \hline
             & 10  &     0.0346 s  &  $76.1$  &  $79.1$ &         & $8.9\times 10^{-9}$  & $1.8\times 10^{-8}$ \\
             & 20  &     0.0606 s  & $136.7$  & $140.6$ &         & $9.6\times 10^{-9}$  & $3.0\times 10^{-8}$ \\
             & 50  &     0.3731 s  & $352.4$  & $357.7$ &         & $9.7\times 10^{-9}$  & $6.0\times 10^{-8}$  \\
GPRP         & 80  &     1.3377 s  & $625.3$  & $631.3$ &         & $9.8\times 10^{-9}$  & $7.8\times 10^{-8}$  \\
             &100  &     2.4659 s  & $753.7$  & $760.1$ &         & $9.8\times 10^{-9}$  & $9.1\times 10^{-8}$  \\
             &150  &     8.3070 s  &$1225.9$  &$1232.9$ &         & $9.9\times 10^{-9}$  & $1.1\times 10^{-7}$  \\
             &200  &     17.208 s  &$1492.9$  &$1500.9$ &         & $9.9\times 10^{-9}$  & $1.2\times 10^{-7}$  \\
             \cline{2-7}\hline
             & 10  &     0.0805 s  & $119.2$  & $121.8$ &         & $8.8\times 10^{-9}$  & $2.6\times 10^{-8}$ \\
             & 20  &     0.1583 s  & $214.4$  & $217.7$ &         & $9.4\times 10^{-9}$  & $4.5\times 10^{-8}$ \\
             & 50  &     0.3781 s  & $308.6$  & $313.6$ &         & $9.4\times 10^{-9}$  & $7.0\times 10^{-8}$  \\
RFR          & 80  &     1.1010 s  & $485.2$  & $490.9$ &         & $9.6\times 10^{-9}$  & $8.5\times 10^{-8}$  \\
             &100  &     1.7944 s  & $523.0$  & $529.0$ &         & $9.7\times 10^{-9}$  & $1.1\times 10^{-7}$  \\
             &150  &     6.7529 s  & $951.2$  & $958.2$ &         & $9.8\times 10^{-9}$  & $1.1\times 10^{-7}$  \\
             &200  &     14.094 s  &$1163.9$  &$1170.9$ &         & $9.9\times 10^{-9}$  & $1.4\times 10^{-7}$  \\
             \cline{2-7}\hline
             & 10  &     0.0422 s  &  $32.6$  &         &         & $4.4\times 10^{-9}$  &   \\
             & 20  &     0.0434 s  &  $30.9$  &         &         & $4.9\times 10^{-9}$  &   \\
Alg.         & 50  &     2.1249 s  & $388.5$  &         &         & $7.4\times 10^{-9}$  &   \\
\ref{ap}     & 80  &     10.209 s  & $766.0$  &         &         & $5.5\times 10^{-9}$  &   \\
             &100  &     17.053 s  & $913.6$  &         &         & $4.2\times 10^{-9}$  &   \\
             &150  &     132.36 s  &$3292.5$  &         &         & $5.1\times 10^{-9}$  &   \\
             &200  &     1349.5 s  & $19111$  &         &         & $0.1681^*$           &   \\
             \cline{2-7}\hline
             & 10  &     0.0078 s  &  5.0     &   6.0   &   16.5  & $1.2\times 10^{-9}$  & $2.7\times 10^{-9}$ \\
             & 20  &     0.0118 s  &  5.6     &   6.6   &   31.2  & $1.8\times 10^{-9}$  & $1.3\times 10^{-8}$  \\
Alg.         & 50  &     0.0550 s  &  6.0     &   7.0   &   52.5  & $1.8\times 10^{-11}$ & $2.3\times 10^{-10}$ \\
\ref{nm1}    & 80  &     0.1907 s  &  6.6     &   7.6   &   70.3  & $1.0\times 10^{-9}$  & $1.2\times 10^{-8}$ \\
             &100  &     0.3634 s  &  6.8     &   7.8   &   80.6  & $1.2\times 10^{-9}$  & $3.2\times 10^{-8}$ \\
             &150  &     0.9421 s  &  7.0     &   8.0   &   98.6  & $3.9\times 10^{-13}$ & $1.2\times 10^{-11}$ \\
             &200  &     1.7102 s  &  7.0     &   8.0   &  105.3  & $1.8\times 10^{-11}$ & $6.1\times 10^{-10}$ \\
             \cline{2-7}\hline
  \end{tabular} }
  \end{center}
\end{table}
\begin{table}[ht]\renewcommand{\arraystretch}{1.2} \addtolength{\tabcolsep}{2pt}
  \caption{Numerical results of Example \ref{ex:2}.}\label{table2}
  \begin{center} {\scriptsize
   \begin{tabular}[c]{|c|c|r|c|c|c|l|l|}
     \hline
Alg.         & $n$ & {\tt CT.} & {\tt IT.} & {\tt NF.} &  {\tt NCG.}   &  {\tt Res.}  &  {\tt grad.}  \\  \hline
             & 10  &     0.0378 s  & $111.2$  & $114.2$ &         & $8.9\times 10^{-9}$  & $1.7\times 10^{-8}$ \\
             & 20  &     0.0861 s  & $185.6$  & $189.5$ &         & $9.5\times 10^{-9}$  & $2.7\times 10^{-8}$ \\
             & 50  &     0.5068 s  & $514.3$  & $519.9$ &         & $9.8\times 10^{-9}$  & $4.4\times 10^{-8}$  \\
GPRP         & 80  &     1.7614 s  & $883.4$  & $889.9$ &         & $9.9\times 10^{-9}$  & $6.6\times 10^{-8}$  \\
             &100  &     3.5525 s  &$1121.1$  &$1128.1$ &         & $9.9\times 10^{-9}$  & $7.0\times 10^{-8}$  \\
             &150  &     11.720 s  &$1709.4$  &$1717.4$ &         & $9.9\times 10^{-9}$  & $9.0\times 10^{-8}$  \\
             &200  &     24.608 s  &$2199.8$  &$2208.7$ &         & $9.9\times 10^{-9}$  & $1.1\times 10^{-7}$  \\
             \cline{2-7}\hline
             & 10  &     0.1182 s  & $188.9$  & $191.2$ &         & $9.3\times 10^{-9}$  & $2.1\times 10^{-8}$ \\
             & 20  &     0.1966 s  & $286.9$  & $290.3$ &         & $9.7\times 10^{-9}$  & $3.5\times 10^{-8}$ \\
             & 50  &     0.5292 s  & $432.9$  & $438.1$ &         & $9.7\times 10^{-9}$  & $6.0\times 10^{-8}$  \\
RFR         & 80  &     1.3037 s  & $583.2$  & $589.1$ &         & $9.7\times 10^{-9}$  & $6.5\times 10^{-8}$  \\
             &100  &     2.2059 s  & $656.1$  & $662.6$ &         & $9.6\times 10^{-9}$  & $1.0\times 10^{-7}$  \\
             &150  &     6.7271 s  & $972.8$  & $980.4$ &         & $9.8\times 10^{-9}$  & $1.3\times 10^{-7}$  \\
             &200  &     14.965 s  &$1205.4$  &$1213.4$ &         & $9.9\times 10^{-9}$  & $1.1\times 10^{-7}$  \\
             \cline{2-7}\hline
             & 10  &     0.0258 s  & $30.4$   &         &         & $6.7\times 10^{-9}$  &   \\
             & 20  &     0.1692 s  & $106.0$  &         &         & $7.8\times 10^{-9}$  &  \\
Alg.         & 50  &     3.5200 s  & $637.0$  &         &         & $7.5\times 10^{-9}$  &  \\
\ref{ap}     & 80  &     2.2655 s  & $250.1$  &         &         & $7.2\times 10^{-9}$  &  \\
             &100  &     47.628 s  &$3019.6$  &         &         & $7.1\times 10^{-9}$  &   \\
             &150  &     647.07 s  & $14994$  &         &         & $0.1546^*$           &   \\
             &200  &     824.65 s  & $11768$  &         &         & $0.1478^*$           &   \\
             \cline{2-7}\hline
             & 10  &     0.0055 s  &  5.2     &   6.2   &   22.6  & $6.4\times 10^{-10}$ & $1.7\times 10^{-9}$ \\
             & 20  &     0.0149 s  &  6.0     &   7.0   &   40.4  & $7.9\times 10^{-14}$ & $3.2\times 10^{-13}$  \\
Alg.         & 50  &     0.0574 s  &  6.0     &   7.0   &   55.9  & $1.7\times 10^{-9}$  & $2.1\times 10^{-8}$ \\
\ref{nm1}    & 80  &     0.2251 s  &  7.0     &   8.0   &   88.9  & $5.1\times 10^{-14}$ & $4.8\times 10^{-13}$ \\
             &100  &     0.4047 s  &  7.0     &   8.0   &   97.3  & $1.7\times 10^{-13}$ & $2.8\times 10^{-12}$ \\
             &150  &     0.9896 s  &  7.0     &   8.0   &  104.8  & $4.6\times 10^{-11}$ & $1.4\times 10^{-9}$ \\
             &200  &     0.9896 s  &  7.1     &   8.1   &  112.5  & $1.7\times 10^{-9}$  & $7.6\times 10^{-8}$ \\
             \cline{2-7}\hline
  \end{tabular} }
  \end{center}
\end{table}

To further illustrate the efficiency of our algorithm, we report the numerical results for Examples \ref{ex:1}--\ref{ex:2} with various problem sizes.
Tables \ref{table3}--\ref{table4} display the numerical results for Examples \ref{ex:1}--\ref{ex:2}.

We see from Tables \ref{table3}--\ref{table4} that Algorithm \ref{nm1}, GPRP and  RFR  work  for  large problems while Algorithm \ref{nm1} is more efficient than GPRP and  RFR for  large problems.

Finally, we point out that all algorithms converge to different solutions for different starting points.

\begin{table}[ht]\renewcommand{\arraystretch}{1.2} \addtolength{\tabcolsep}{2pt}
  \caption{Numerical results of Example \ref{ex:1}.}\label{table3}
  \begin{center} {\scriptsize
   \begin{tabular}[c]{|c|r|r|c|c|c|c|l|}
     \hline
Alg.         & $n$  & {\tt CT.} & {\tt IT.} & {\tt NF.} &  {\tt NCG.}   &   {\tt Res.}    &  {\tt grad.}  \\  \hline
             & 400  &      04 m 17 s  &  $3074$  &  $3083$ &         & $9.9\times 10^{-9}$  & $9.0\times 10^{-8}$ \\
GPRP         & 600  &      14 m 11 s  &  $3641$  &  $3650$ &         & $9.9\times 10^{-9}$  & $2.5\times 10^{-7}$ \\
             & 800  &      41 m 47 s  &  $5108$  &  $5118$ &         & $9.9\times 10^{-9}$  & $4.2\times 10^{-7}$  \\
             &1000  & 01 h 14 m 23 s  &  $5420$  &  $5430$ &         & $9.9\times 10^{-9}$  & $5.0\times 10^{-7}$  \\
             \cline{2-7}\hline
             & 400  &      04 m 36 s  &  $3324$  &  $3332$ &         & $9.9\times 10^{-9}$  & $1.9\times 10^{-7}$ \\
RFR          & 600  &      24 m 23 s  &  $5840$  &  $5849$ &         & $9.9\times 10^{-9}$  & $1.9\times 10^{-7}$ \\
             & 800  & 01 h 07 m 53 s  &  $8087$  &  $8096$ &         & $9.9\times 10^{-9}$  & $2.5\times 10^{-7}$  \\
             &1000  & 02 h 38 m 59 s  & $11157$  & $11167$ &         & $9.9\times 10^{-9}$  & $2.6\times 10^{-7}$  \\
             \cline{2-7}\hline
             & 400  &         23.7 s  &   8.0    &   9.0   &   166.9   & $3.2\times 10^{-13}$ & $7.1\times 10^{-12}$ \\
Alg.         & 600  &      01 m 05 s  &   8.0    &   9.0   &   169.5   & $3.4\times 10^{-12}$ & $2.7\times 10^{-10}$  \\
\ref{nm1}    & 800  &      02 m 23 s  &   8.0    &   9.0   &   162.4   & $8.3\times 10^{-9}$  & $1.7\times 10^{-6}$ \\
             &1000  &      07 m 09 s  &   9.0    &  10.0   &   229.3   & $1.2\times 10^{-12}$ & $4.2\times 10^{-11}$ \\
             \cline{2-7}\hline
\end{tabular} }
  \end{center}
\end{table}
\begin{table}[ht]\renewcommand{\arraystretch}{1.2} \addtolength{\tabcolsep}{2pt}
  \caption{Numerical results of Example \ref{ex:2}.}\label{table4}
  \begin{center} {\scriptsize
   \begin{tabular}[c]{|c|r|r|c|c|c|c|l|}
     \hline
Alg.         & $n$  & {\tt CT.} & {\tt IT.} & {\tt NF.} &  {\tt NCG.}   &   {\tt Res.}    &  {\tt grad.}  \\  \hline
             & 400  &      06 m 30 s  &  $4805$  &  $4815$ &         & $9.9\times 10^{-9}$  & $1.4\times 10^{-7}$ \\
GPRP         & 600  &      17 m 45 s  &  $4127$  &  $4139$ &         & $9.9\times 10^{-9}$  & $3.8\times 10^{-7}$ \\
             & 800  &      42 m 50 s  &  $5283$  &  $5295$ &         & $9.9\times 10^{-9}$  & $4.4\times 10^{-7}$  \\
             &1000  & 01 h 17 m 15 s  &  $5421$  &  $5434$ &         & $9.9\times 10^{-9}$  & $5.3\times 10^{-7}$  \\
             \cline{2-7}\hline
             & 400  &      04 m 46 s  &  $3820$  &  $3830$ &         & $9.9\times 10^{-9}$  & $1.3\times 10^{-7}$ \\
RFR          & 600  &      25 m 53 s  &  $6243$  &  $6254$ &         & $9.9\times 10^{-9}$  & $1.5\times 10^{-7}$ \\
             & 800  & 01 h 19 m 33 s  &  $9294$  &  $9306$ &         & $9.9\times 10^{-9}$  & $2.1\times 10^{-7}$  \\
             &1000  & 02 h 50 m 14 s  & $11861$  & $11873$ &         & $9.9\times 10^{-9}$  & $2.8\times 10^{-7}$  \\
             \cline{2-7}\hline
             & 400  &         24.0 s  &   8.0    &   9.0   &   183.9   & $6.0\times 10^{-13}$ & $3.2\times 10^{-11}$ \\
Alg.         & 600  &      01 m 02 s  &   8.0    &   9.0   &   171.6   & $1.1\times 10^{-9}$  & $1.5\times 10^{-7}$  \\
\ref{nm1}    & 800  &      03 m 41 s  &   9.0    &  10.0   &   224.4   & $1.4\times 10^{-12}$ & $1.5\times 10^{-10}$ \\
             &1000  &      06 m 22 s  &   9.0    &  10.0   &   224.0   & $8.2\times 10^{-11}$ & $2.3\times 10^{-8}$ \\
             \cline{2-7}\hline
  \end{tabular} }
  \end{center}
\end{table}

\section{Conclusions}\label{sec7}
This paper is concerned with the nonnegative inverse eigenvalue problem. The inverse problem is rewritten as an under-determined constrained nonlinear matrix equation over several matrix manifolds. Then a Riemannian inexact Newton-CG method is proposed for solving the constrained nonlinear matrix equation. The global and quadratic convergence of the proposed geometric method is established under some mild conditions. Our  method is also extended to the case of prescribed entries.  Numerical tests illustrate the efficiency of the proposed geometric algorithm. From our numerical tests, we observe that, for large problems, most  of our computing time is spent on the CG method for solving  (\ref{eq:le}). It would improve the efficiency if one can find a good preconditioner for (\ref{eq:le}), which needs further study.




\vskip 5mm  {\bf Appendix A.} In this appendix, we establish some basic properties of the product manifold $\Rnn\times \co(n)\times \cv$ and the differential of $G$ defined in (\ref{eq:NIEP1}). We first show that  the nonlinear matrix equation (\ref{eq:NIEP1}) is under-determined for all $n\ge 2$. The  dimension of $\Rnn\times \co(n)\times \cv$ is given by
\[
\dim (\Rnn \times \co(n) \times \cv) =  n^2 + \frac{n(n-1)}{2} + |\mathcal{J}|,
\]
where $\mathcal{J}$ is the complementary index set of $\mathcal{I}$ with respect to the index set $\cn$, and $|\mathcal{J}|$ is the cardinality of $\mathcal{J}$.
Thus
\[
\dim (\Rnn \times \co(n) \times \cv) >  \dim \Rnn\quad\mbox{for } n\ge 2.
\]
Hence, (\ref{eq:NIEP1}) is under-determined for all $n\ge 2$.

The tangent space of $\Rnn \times \co(n) \times \cv$ at a point $(S,Q,V)\in\Rnn \times \co(n) \times \cv$ is given by
\[
\begin{array}{c}
T_{(S,Q,V)}\big( \Rnn \times\co(n)\times \cv\big) = T_S \Rnn\times T_Q \co(n)\times T_V \cv.
\end{array}
\]
Here, $T_S \Rnn$, $T_Q \co(n)$, and $T_V \cv$ are the tangent spaces of $\Rnn$, $\co(n)$, and $\cv$ at $S\in\Rnn$, $Q\in\co(n)$, and $V\in\cv$ accordingly, which are
given by \cite[p.42]{AMS08}:
\[
T_S \Rnn =\Rnn,\quad
T_Q \co(n)=  \big\{ Q\Omega \ | \  \Omega^T = -\Omega ,\; \Omega \in \Rnn \big\}, \quad
T_V \cv = \cv.
\]

A retraction $R$ on $\Rnn\times\co(n)\times \cv$ is given by
\[
R_{(S,Q,V)} (\xi_S,\zeta_Q,\eta_V) = \big(R_S(\xi_S), R_Q(\zeta_Q),R_{V}(\eta_V)\big)
\]
for all  $(S,Q,V) \in \Rnn \times \co(n) \times \cv$ and $(\xi_S,\eta_Q,\gamma_V)\in T_{(S,Q,V)}\big( \Rnn\times\co(n)\times \cv\big)$, where $R_S$, $R_Q$, and $R_{V}$ are the retractions on $\Rnn$, $\co(n)$, and $\cv$  accordingly,  which may take the following form:
\[
\left\{
\begin{array}{ccl}
R_S(\xi_S) &=& S + \xi_S,  \quad {\rm for} \; \xi_S \in  T_S\Rnn,\\[1.5mm]
R_Q(\zeta_Q) &=& {\rm qf} (Q + \zeta_Q),  \quad {\rm for} \; \zeta_Q \in  T_Q\co(n),\\[1.5mm]
R_{V}(\eta_V) &=& V+\eta_V, \quad {\rm for} \; \eta_V \in T_{V} \cv.
\end{array}
\right.
\]
Here, ${\rm qf}(A)$ means the $Q$ factor of the QR decomposition of a nonsingular matrix $A\in \Rnn$ in the form of $A=Q\widetilde{R}$ with $Q\in\co(n)$ and $\widetilde{R}$ being an upper triangular matrix with strictly positive diagonal entries. For other choices of retractions on  $\co(n)$, one may refer to \cite[pp.58--59]{AMS08}.

We now establish the differential of $G$. By simple calculation, the differential $\mathrm{D}G(S,Q,V): T_{(S,Q,V)}\big(\Rnn \times \co(n) \times \mathcal{V}\big) \to T_{G(S,Q,V)}\Rnn\simeq\Rnn$ of $G$ at $(S,Q,V) \in \Rnn \times \co(n) \times \cv$  is determined by
\[
\mathrm{D}G(S,Q,V) [(\Delta S,\Delta Q, \Delta V)] = 2S\odot \Delta S + [ Q(\Lambda+V)Q^T , \Delta QQ^T]- Q\Delta VQ^T
\]
for all $(\Delta S,\Delta Q, \Delta V) \in T_{(S,Q,V)}(\Rnn \times \co(n) \times \mathcal{V})$.
On the other hand, with respect to the Riemannian metric $\langle\cdot,\cdot\rangle$,  the adjoint $(\mathrm{D}G(S,Q,V))^* : T_{G(S,Q,V)}\Rnn \to T_{(S,Q,V)}(\Rnn \times\co(n)\times \mathcal{V})$ of
$\mathrm{D}G(S,Q,V)$ is determined by
\[
(\mathrm{D}G(S,Q,V))^* [ \Delta Z ]
= ((\mathrm{D}G(S,Q,V))_1^* [ \Delta Z ],(\mathrm{D}G(S,Q,V))_2^* [ \Delta Z ],(\mathrm{D}G(S,Q,V))_3^* [ \Delta Z ])
\]
for all $\Delta Z\in T_{G(S,Q,V)}\Rnn$ and for each $\Delta Z\in T_{G(S,Q,V)}\Rnn$,
\[
\left\{
\begin{array}{rcl}
(\mathrm{D}G(S,Q,V))_1^* [ \Delta Z ] &=& 2S\odot \Delta Z, \\[2mm]
(\mathrm{D}G(S,Q,V))_2^* [ \Delta Z ] &=& \displaystyle \frac{1}{2}\big( [ Q(\Lambda + V)Q^T, (\Delta Z)^T ]
+ [ Q(\Lambda + V)^TQ^T, \Delta Z ]\big)Q, \\[2mm]
(\mathrm{D}G(S,Q,V))_3^* [ \Delta Z ] &=& -W\odot\big(Q^T\Delta Z Q\big),
\end{array}
\right.
\]
where $W\in \Rnn$ is defined by $W_{ij} =0$,  if $ (i,j) \in \mathcal{I}$; $1$,  otherwise.

\vskip 5mm  {\bf Appendix B.}  In this appendix, we establish some basic properties of the product manifold $\cz\times \co(n)\times \cv$  and the differential of $H$ defined in (\ref{NIEP-PE}). First, the tangent space of $\cz\times \co(n)\times \cv$ at a point $(S,Q,V)\in\cz\times \co(n)\times \cv$ is given by
\[
T_{(S,Q,V)}(\cz\times \co(n)\times \cv) = T_S\cz\times T_Q\co(n)\times T_V\cv,
\]
where $T_S\cz=\cz$ and  $T_Q\co(n)$ and $T_V\cv$ are defined as in Appendix A.

A retraction $R$ on $\cz\times\co(n)\times \cv$ takes the form of
\[
R_{(S,Q,V)} (\xi_S,\zeta_Q,\eta_V) = \big(R_S(\xi_S), R_Q(\zeta_Q),R_{V}(\eta_V)\big)
\]
for all $(\xi_S,\eta_Q,\gamma_V)\in T_{(S,Q,V)}\big( \cz\times\co(n)\times \cv\big)$, where $R_S(\xi_S)=S+\xi_S$ for $\xi_S\in T_S\cz$ and $R_Q(\zeta_Q)$ and $R_{V}(\eta_V)$ are defined as in Appendix A.

Next, we establish the differential of  $H$. By simple calculation, the differential $\mathrm{D}H(S,Q,V):$ $T_{(S,Q,V)}\big(\cz\times\co(n)\times \cv\big)$ $\to T_{H(S,Q,V)}\Rnn$ $\simeq\Rnn$ of $H$ at a point $(S,Q,V) \in \cz\times \co(n)\times \cv$ is determined by
\[
\begin{array}{lcl}
\mathrm{D}H(S,Q,V) [(\Delta S, \Delta Q, \Delta V)]
= 2S\odot \Delta S + [ Q(\Lambda+V)Q^T , \Delta QQ^T]- Q\Delta VQ^T
\end{array}
\]
for all $(\Delta S,\Delta Q, \Delta V) \in T_{(S,Q,V)}(\cz\times \co(n)\times \cv)$. Let  $\cz\times \co(n)\times \cv$ be equipped with the Riemannian metric defined as in (\ref{def:ip}). Then the adjoint $(\mathrm{D}H(S,Q,V))^*:$\\ $T_{H(S,Q,V)}\Rnn\to T_{(S,Q,V)}\big(\cs\times \co(n)\times \cv\big)$ is given by
\[
(\mathrm{D}H(S,Q,V))^* [ \Delta Z ]
:=((\mathrm{D}H_1(S,Q,V))^* [ \Delta Z ],(\mathrm{D}H_2(S,Q,V))^* [ \Delta Z ],(\mathrm{D}H_3(S,Q,V))^* [ \Delta Z ])
\]
for all $\Delta Z \in T_{\Phi(S,Q,V)}\Rnn\simeq\Rnn$, where for each $\Delta Z \in T_{\Phi(S,Q,V)}\Rnn$,
\[
\left\{
\begin{array}{lcl}
(\mathrm{D}H_1(S,Q,V))^* [ \Delta Z ] &=& 2S\odot \Delta Z, \\[2mm]
(\mathrm{D}H_2(S,Q,V))^* [ \Delta Z ] &=&\displaystyle  \frac{1}{2}\big([ Q(\Lambda + V)Q^T, (\Delta Z)^T ] + [ Q(\Lambda + V)^TQ^T, \Delta Z ]\big)Q, \\[2mm]
(\mathrm{D}H_3(S,Q,V))^* [ \Delta Z ] &=& -W\odot\big(Q^T\Delta Z Q\big).
\end{array}
\right.
\]

\end{document}